\documentclass[11pt]{article}

\usepackage{calc}
\usepackage{tikz}
 	 \usetikzlibrary{arrows,backgrounds} 
 	 \usetikzlibrary{calc}
\tikzstyle{shaded}=[fill=red!10!blue!20!gray!30!white]
\tikzstyle{unshaded}=[fill=white]
\tikzstyle{empty box}=[circle, draw, thick, fill=white, opaque, inner sep=2mm]
\tikzstyle{annular}=[scale=.65, inner sep=1mm, baseline=-.1cm]
\tikzstyle{rectangular}=[scale=.75, inner sep=1mm, baseline=-.2cm]

\usepackage{dbnsymb} 
\usetikzlibrary{knots}
\usepackage{xcolor} 
\usepackage{longtable}

\usetikzlibrary{decorations.markings}
\tikzstyle{vertex}=[circle, draw, inner sep=0pt, minimum size=5pt]
\tikzstyle{vertice}=[circle, draw, inner sep=0pt, minimum size=5pt, fill=black]

\newcommand{\vertex}{\node[vertex]}
\newcommand{\vertice}{\node[vertice]}
\newcommand{\vin}{\rotatebox[origin=c]{180}{$\overcrossing$}}

\usepackage{amsmath,amsthm,amsfonts,latexsym,tikz,mathabx,stmaryrd,esvect}

\begin{document}

\newtheorem*{theo}{Theorem}
\newtheorem*{pro} {Proposition}
\newtheorem*{cor} {Corollary}
\newtheorem*{lem} {Lemma}
\newtheorem{theorem}{Theorem}[section]
\newtheorem{corollary}[theorem]{Corollary}
\newtheorem{lemma}[theorem]{Lemma}
\newtheorem{proposition}[theorem]{Proposition}
\newtheorem{conjecture}[theorem]{Conjecture}
\newtheorem{question}{Question}

\theoremstyle{definition}
 \newtheorem{definition}[theorem]{Definition} 
  \newtheorem{example}[theorem]{Example}
   \newtheorem{remark}[theorem]{Remark}
   
\newcommand{\Naturali}{{\mathbb{N}}}
\newcommand{\Reali}{{\mathbb{R}}}
\newcommand{\Complessi}{{\mathbb{C}}}
\newcommand{\Toro}{{\mathbb{T}}}
\newcommand{\Relativi}{{\mathbb{Z}}}
\newcommand{\HH}{\mathfrak H}
\newcommand{\KK}{\mathfrak K}
\newcommand{\LL}{\mathfrak L}
\newcommand{\as}{\ast_{\sigma}}
\newcommand{\tn}{\vert\hspace{-.3mm}\vert\hspace{-.3mm}\vert}
\def\A{{\cal A}}
\def\B{{\cal B}}
\def\E{{\cal E}}
\def\F{{\cal F}}
\def\G{{\cal G}}
\def\H{{\cal H}}
\def\K{{\cal K}}
\def\L{{\cal L}}
\def\N{{\cal N}}
\def\M{{\mathcal M}}
\def\gM{{\frak M}}
\def\O{{\cal O}}
\def\P{{\cal P}}
\def\S{{\cal S}}
\def\T{{\mathcal T}}
\def\U{{\cal U}}
\def\V{{\mathcal V}}
\def\qed{\hfill$\square$}

\title{
The Jones polynomial
and functions of positive type  
on the oriented Jones-Thompson groups $\vec{F}$ and $\vec{T}$
}

\author{Valeriano Aiello, Roberto Conti\\}
\date{\today}
\maketitle
\markboth{V. Aiello, R. Conti}{
}
\renewcommand{\sectionmark}[1]{}
\begin{abstract}
The pioneering work of Jones and Kauffman unveiled 
a fruitful relationship between statistical mechanics and knot theory. Recently, Jones introduced two subgroups $\vec{F}$ and $\vec{T}$ of the Thompson groups $F$ and $T$, respectively, together with a procedure that associates an oriented link diagram to any element of these subgroups. 
Moreover, 
several specializations of some well-known polynomial link invariants can be seen as functions of positive type on the Thompson groups or the Jones-Thompson subgroups. 
One important example is provided by
suitable evaluations of the Jones polynomial, which are thus associated with certain unitary representations of the groups $\vec{F}$ and $\vec{T}$.  
Within this framework, we discuss an alternative approach that relies on some partition function interpretation of the Jones polynomial, and also exhibit more examples associated 
with other link invariants, notably the two-variable Kauffman polynomial and the HOMFLY polynomial.
In the unoriented case, extending our previous results, we also show by similar methods that certain evaluations of the Tutte polynomial and of the Kauffman bracket,
 suitably renormalized,
yield functions of positive type on $T$. 


\end{abstract}

\newpage

\section{Introduction}

In \cite{Jo} it is introduced the oriented subgroup $\vec{F}$ of R. Thompson group $F$ consisting of those elements for which the associated $\Gamma$-graph (as described therein) is bipartite. In turn, this is also the subgroup of $F$ of elements for which there exists a canonical orientation on the associated link diagram $\L(g)$ by means of Seifert surfaces (as explained in \cite{Jo}, $\L(g)$ is uniquely defined up to distant unknots arising from opposite carets). It is worth to mention that $\vec{F}$ is isomorphic to the Thompson group $F_3$ defined using triadic subdivisions \cite{Golan}.
In the sequel we will denote the oriented link diagram associated to $g \in \vec{F}$ by $\vec{\L}(g)$.
Therefore, it is meaningful to compute invariants of oriented links on the elements of $\vec{F}$, after a suitable normalization to get rid of the distant unknots.
We thus introduce the Jones function $V_g(t)$, that is equal to the Jones polynomial $V_{\vec{\L}(g)}(t)$ up to a multiplicative factor ensuring its 
well-definiteness as a function on $\vec{F}$.
This path was of course initiated by Jones in \cite{Jo}, where he showed that many specializations of some 
familiar polynomial invariants, including the Jones polynomial and the Kauffman bracket, can be seen as coefficients of unitary representations of the relevant groups involved.

In \cite{AiCo1} we gave a somehow different proof of the fact that the renormalized Kauffman bracket, agreeing with $\langle \L(g) \rangle (A)$ up to a suitable factor, provides a function of positive type on $F$ for certain special values of the variable $A$, which actually form a finite set of roots of unity.
In this paper, by using the  close relationship between the Kauffman bracket and the Jones polynomial, we present an alternative proof of the fact that also the Jones function $V_g(t)$ gives rise to functions of positive type on $\vec{F}$, at least for some specific values of the variable $t$.
A  bit unexpectedly, the evaluations selected in this way correspond to specific values of the Jones polynomial that have already attracted some interest in the literature, mainly because they admit a geometrical interpretation.
Moreover, proceeding along similar lines, 
we are able to extend the arguments to cover also some specializations of the (renormalized) two-variable Kauffman polynomial and (some functions associated to) the Homfly polynomial. 

A major ingredient in the proof of these results is the possibility to represent the invariants as partition functions of suitably defined statistical mechanical models.
We believe this approach has the merit of being accessible to a wide, non-specialized audience, the drawback being that the relevant specializations of the variables 
obtained in this way are far from being optimal (see \cite{ACJ} for detailed statements).
As a byproduct of this analysis, we also make a number of observations concerning the specific topological properties (writhe, number of components) of the links obtained with the Jones procedure that might be of independent interest.

It is also natural to investigate analogous problems for the larger groups $\vec{T} \subset T$.
We show that all the results about the Tutte polynomial and the Kauffman bracket obtained in \cite{AiCo1} remain true without modifications 
for $T$, as well as those 
about the Jones polynomial for $\vec{T}$.
A further and more general 
analysis regarding oriented link invariants giving rise to functions of positive type on $\vec{F}$ is pursued in 
\cite{ACJ} by using the recent machinery developed in \cite{Jo-No}.

In a continuation of this project, 
one might also 
determine which evaluations of the Jones function 
(or of some 
closely related function) belong to the linear span of the
functions of positive type, namely they lie in the Fourier-Stieltjes algebra $B(\vec{F})$, or even in some other space of multipliers.
All in all, one might wonder whether this circle of ideas from low-dimensional topology may provide any valuable information on the various multiplier spaces of $\vec{F}$ (and $F$). We 
plan to return to these issues elsewhere.

\section{Preliminaries}
This section is devoted to introducing the definition of the Jones-Thompson group $\vec{F}$ and recalling the fundamental properties of some polynomial invariants for oriented links. In the first place (see \cite{Cannon}) we recall that the group $T$ can be defined as the group generated by two families of elements $\{x_n\}_{n\geq 0}$ and $\{c_n\}_{n\geq 0}$ such that
\begin{enumerate} 
\item $x_j x_i=x_ix_{j+1}$ if $i<j$ 
\item $x_k c_{n+1}=c_n x_{k+1}$ if $k<n$  
\item $c_n x_0 = c_{n+1}^2$ 
\item $c_n = x_n c_{n+1}$ 
\item $c_n^{n+2} =1$ . 
\end{enumerate}
The group $F$ is the subgroup of $T$ generated by the $\{x_n\}_{n\geq 0}$.

In \cite[Section 4]{Jo}, Jones associates to any element $g \in T$, represented by a pair of bifurcating trees without cancelling carets, a certain graph $\Gamma(g)$. The Jones-Thompson group can then be defined as follows:
$$
\vec{F}:=\{g \in F \ | {\rm Chr}_{\Gamma(g)}(2) = 2 \} \ , 
$$
where ${\rm Chr}_\Gamma$ denotes the chromatic polynomial of the graph $\Gamma$.
Golan and Sapir proved that this group is finitely generated exhibiting some generators (\cite{Golan}, Lemma 4.6), namely $x_0x_1$, $x_1x_2$ and $x_2x_3$. Moreover, using also the fact that these elements satisfy the defining relations of the generators of $F_3$ they actually showed that indeed $\vec{F}$ is isomorphic to the Thompson group $F_3$ \cite[Lemma 4.7]{Golan}. 

Using $\Gamma(g)$, Jones (\cite{Jo}) discovered a way to  associate an oriented link to any element of the Jones-Thompson group. We briefly outline this procedure in the following example.

\begin{example}
Being $\vec{F}$ a subgroup of $F$, any element of $\vec{F}$ can be represented by a pair of trees.  
Consider the following element of $\vec{F}$, the associated $\Gamma$-graph and the corresponding medial graph
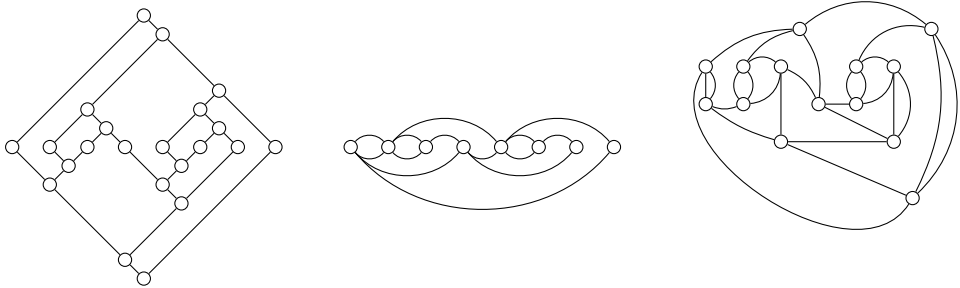
\[\begin{tikzpicture}[x=0.5cm, y=0.5cm,
    every edge/.style={
        draw,
      postaction={decorate,
                    decoration={markings}
                   }
        }
]
	\vertex (a) at (0,0) {};
	\vertex (b) at (1,0) {};
	\vertex (c) at (2,0) {};
	\vertex (d) at (3,0) {};
	\vertex (e) at (4,0) {};
	\vertex (f) at (5,0) {};
	\vertex (g) at (6,0) {};
	\vertex (h) at (7,0) {};
	\vertex (i) at (5.5,0.5) {};
	\vertex (l) at (5,1) {};;
	\vertex (m) at (5.5,1.5) {};
	\vertex (n) at (2.5,0.5) {};
	\vertex (o) at (2,1) {};
	\vertex (p) at (4,3) {};
	\vertex (q) at (3.5,3.5) {};
	\vertex (r) at (1.5,-0.5) {};
	\vertex (s) at (1,-1) {};
	\vertex (t) at (4,-1) {};
	\vertex (u) at (4.5,-0.5) {};
	\vertex (v) at (4.5,-1.5) {};
	\vertex (w) at (3,-3) {};
	\vertex (z) at (3.5,-3.5) {};
	\path
		(a) edge (q)
		(b) edge (o)
		(o) edge (p)
		(p) edge (q)
		(c) edge (n)
		(n) edge (o)
		(d) edge (n)
		(e) edge (l)
		(l) edge (m)
		(m) edge (p)
		(f) edge (i)
		(i) edge (l)
		(g) edge (i)
		(h) edge (m)
		(a) edge (s)
		(s) edge (w)
		(w) edge (z)
		(b) edge (r)
		(r) edge (s)
		(c) edge (r)
		(d) edge (t)
		(t) edge (v)
		(v) edge (w)
		(e) edge (u)
		(u) edge (t)
		(f) edge (u)
		(g) edge (v)
		(h) edge (z);
	\vertex (a) at (9,0) {};
	\vertex (b) at (10,0) {};
	\vertex (c) at (11,0) {};
	\vertex (d) at (12,0) {};
	\vertex (e) at (13,0) {};
	\vertex (f) at (14,0) {};
	\vertex (g) at (15,0) {};
	\vertex (h) at (16,0) {};
	\path
		(a) edge[bend left=50] (b)
		(b) edge[bend left=50] (c)
		(b) edge[bend left=50] (e)
		(c) edge[bend left=50] (d)
		(e) edge[bend left=50] (f)
		(f) edge[bend left=50] (g)
		(e) edge[bend left=50] (h)
		(a) edge[bend right=50] (b)
		(b) edge[bend right=50] (c)
		(a) edge[bend right=50] (d)
		(d) edge[bend right=50] (e)
		(e) edge[bend right=50] (f)
		(d) edge[bend right=50] (g)
		(a) edge[bend right=50] (h);
\end{tikzpicture}
\quad
\begin{tikzpicture}[x=0.5cm, y=0.5cm,
    every edge/.style={
        draw,
      postaction={decorate,
                    decoration={markings}
                   }
        }
]
	\vertex (a) at (0,0) {};
	\vertex (b) at (1,0) {};
	\vertex (c) at (2,0) {};
	\vertex (d) at (2.5,1) {};
	\vertex (e) at (4,0) {};
	\vertex (f) at (5,0) {};
	\vertex (g) at (6,1) {};
	\vertex (h) at (0,-1) {};
	\vertex (i) at (1,-1) {};
	\vertex (l) at (2,-2) {};
	\vertex (m) at (3,-1) {};
	\vertex (n) at (4,-1) {};
	\vertex (o) at (5,-2) {};
	\vertex (p) at (5.5,-3.5) {};
		
	\path
		(a) edge[bend left=40] (h)
		(a) edge (h)
		(b) edge[bend left=40] (i)
		(b) edge[bend right=40] (i)
		(e) edge[bend left=40] (n)
		(n) edge[bend left=40] (e)
		(b) edge[bend left=40] (c)
		(i) edge[bend right=40] (c)
		(h) edge[bend right=20] (i)
		(h) edge[bend right=10] (l)
		(l) edge (c)
		(a) edge[bend left=20] (d)
		(b) edge[bend left=20] (d)
		(c) edge[bend left=20] (m)
		(d) edge[bend left=20] (m)
		(m) edge (n) 
		(m) edge (o)
		(n) edge[bend right=40] (f)
		(e) edge[bend left=40] (f)
		(f) edge (o)
		(o) edge[bend right=40] (f)
		(g) edge[bend left=20] (p)
		(p) edge[bend right=40] (g)
		(d) edge[bend left=40] (g)
		(e) edge[bend left=40] (g)
		(l) edge (o)
		(l) edge (p)
		(a) edge[bend left=270] (p)
;
\end{tikzpicture}\]
Indeed, this is an element of $\vec{F}$ as it can be easily seen that the associated $\Gamma$-graph is bipartite.
The corresponding 
link diagram, depicted below, can be recognized to represent an oriented Hopf link.
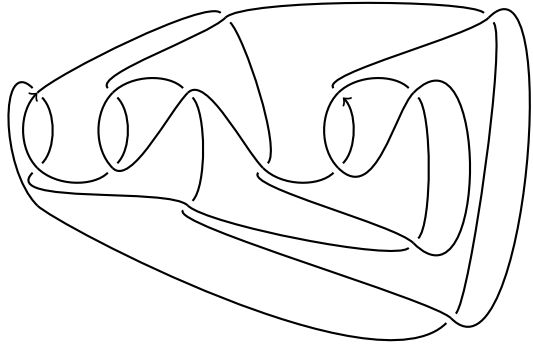
\[\begin{tikzpicture}[every path/.style={thick}, every node/.style={transform shape, knot crossing, inner sep=1.5pt}]
	\node (a) at (0,0) {};
	\node (b) at (1,0) {};
	\node (c) at (2,0) {};
	\node (d) at (2.5,1) {};
	\node (e) at (4,0) {};
	\node (f) at (5,0) {};
	\node (g) at (6,1) {};
	\node (h) at (0,-1) {};
	\node (i) at (1,-1) {};
	\node (l) at (2,-1.5) {};
	\node (m) at (3,-1) {};
	\node (n) at (4,-1) {};
	\node (o) at (5,-2) {};
	\node (p) at (5.5,-3) {};
	\node (x) at (0,-1.5) {};

\draw[->] (h.center) .. controls (h.4 north west) and (a.4 south west) ..  (a.center);
\draw (a.center) .. controls (a.4 north east) and (d.4 north west) ..  (d);
\draw (d) .. controls (d.4 south east) and (m.4 north east) ..  (m);
\draw (m) .. controls (m.4 south west) and (o.4 north west) ..  (o.center);
\draw (o.center) .. controls (o.16 south east) and (f.16 north east) ..  (f.center);
\draw (f.center) .. controls (f.4 south west) and (n.8 south east) ..  (n.center);
\draw (n.center) .. controls (n.4 north west) and (e.4 south west) ..  (e.center);
\draw (e.center) .. controls (e.4 north east) and (f.4 north west) ..  (f);
\draw (f) .. controls (f.4 south east) and (o.4 north east) ..  (o);
\draw (o) .. controls (o.4 south west) and (l.4 south east) ..  (l.center);
\draw (l.center) .. controls (l.4 north west) and (h.8 south west) ..  (h);
\draw (h) .. controls (h.4 north east) and (a.4 south east) ..  (a);
\draw (a) .. controls (a.8 north west) and (x.8 north west) .. (x.center); 
\draw (x.center) .. controls (x.4 south east) and (p.16 south west) ..  (p);
\draw (p) .. controls (p.4 north east) and (g.4 south east) ..  (g);
\draw (g) .. controls (g.4 north west) and (d.4 north east) ..  (d.center);
\draw (d.center) .. controls (d.4 south west) and (b.4 north west) ..  (b);
\draw (b) .. controls (b.4 south east) and (i.4 north east) ..  (i);
\draw (i) .. controls (i.4 south west) and (h.4 south east) ..  (h.center);
\draw (i.center) .. controls (i.4 north west) and (b.4 south west) ..  (b.center);
\draw (b.center) .. controls (b.4 north east) and (c.4 north west) ..  (c);
\draw (c) .. controls (c.4 south east) and (l.4 north east) ..  (l);
\draw (l) .. controls (l.4 south west) and (p.4 north west) ..  (p.center);
\draw (p.center) .. controls (p.16 south east) and (g.16 north east) ..  (g.center); jkjhkjhkjh
\draw (g.center) .. controls (g.4 south west) and (e.4 north west) ..  (e);
\draw[<-] (e) .. controls (e.4 south east) and (n.4 north east) ..  (n);
\draw (n) .. controls (n.4 south west) and (m.4 south east) ..  (m.center);
\draw (m.center) .. controls (m.4 north west) and (c.4 north east) ..  (c.center);
\draw (c.center) .. controls (c.4 south west) and (i.4 south east) ..  (i.center);
%
%
%
%
%
%
%
%
%
%
\end{tikzpicture}\] 
\end{example}

Now we consider some polynomial invariants of oriented links. The Homfly polynomial and Jones polynomial are Laurent polynomial satisfying the so-called skein relations, namely 
$$
		 \left\{\begin{array}{l}
 		t^{-1}V_{\overcrossing}(t)-tV_{\undercrossing}(t)=(t^{1/2}-t^{-1/2})V_{\upuparrows}\\
 		\alpha P_{\overcrossing}(\alpha,z)-\alpha^{-1}P_{\undercrossing}(\alpha,z)-z P_{\upuparrows}(\alpha,z)=0.
 		\end{array}
 		\right.	
$$
The $2$-variable Kauffman polynomial, the Homfly polynomial and the Jones polynomial satisfy the following properties
\begin{align*}
 & \left\{\begin{array}{l}
 	F_{\vec{L}_1\cup \vec{L}_2}= [(a^{-1}+a)x^{-1}-1]F_{\vec{L}_1}F_{\vec{L}_2}\\
F_O= 1\\
 		\end{array}
 		\right. \\
& \left\{\begin{array}{l}
 		P_{\vec{L_1}\cup \vec{L_2}}=\frac{(\alpha-\alpha^{-1})}{z} P_{\vec{L_1}}P_{\vec{L_2}} ,\\
P_O= 1\\
 		\end{array}
 		\right. \\
 & \left\{\begin{array}{l}
V_{\vec{L}\cup O}=(-t^{-1/2}-t^{1/2})V_{\vec{L}},\\
V_O= 1. 		
\end{array}
 		\right.
\end{align*}
For $g \in \vec{F}$, represented by a pair of trees $(T_+,T_-)$ with $n$ leaves,
we define the $2$-variables Kauffman function, the Homfly function and the Jones function, respectively, by
\begin{eqnarray*}
F_g(a,x) & := & F_{\vec{\L}(T_+,T_-)}(a,x) [(a^{-1}+a)x^{-1}-1]^{-n+1} \ , \\
P_g(\alpha,z) & := & P_{\vec{\L}(T_+,T_-)}(\alpha,z) [(\alpha-\alpha^{-1})/z]^{-n+1} ,\\
V_g(t) & := & V_{\vec{\L}(T_+,T_-)}(t)(-t^{-1/2}-t^{1/2})^{-n+1} \ . 
\end{eqnarray*}
Therefore, the maps $g \mapsto F_g(a,x)$, $g\mapsto P_g(\alpha,z)$ and $g \mapsto V_g(t)$ are well-defined on $\vec{F}$ for all the values of the involved parameters.

In the sequel, we will repeatedly use the well-known fact that a matrix $A\in M_r(\mathbb{C})$ is positive if and only if there exist a Hilbert space $H$ and vectors $v_1,\ldots, v_r\in H$ such that $(a_{ij})=(\langle v_i,v_j\rangle)$, see e.g. \cite{Bhatia}, p. 2.

\section{On the Jones polynomial as a function on $\vec{F}$ }

Building upon our previous work \cite{AiCo1}, 
we now focus on the problem of 
finding some specializations of the variable for which the Jones polynomial 
gives rise to a positive type functions on $\vec{F}$. 
A first result in this direction comes from the notion of 3-colourings (see \cite{encycl}, p. 162).


We proved in \cite{AiCo1} that the number of $n$-colourings ${\rm Col}_n$ defines a positive definite function on $F$.
Moreover, it has been observed in \cite[Theorem 1.13]{P} that 
$${\rm Col}_3(L) = 3 \left | V_{\vec{L}}(e^{i \pi /3}) \right |^2 = 3 |F_{\vec{L}}(1,-1)| \ , $$
where $F_{\vec{L}}(a,x)$ denotes the two-variable Kauffman polynomial.
Therefore, we immediately obtain the following result.

\begin{proposition}\label{threecol}
The functions $| V_g (e^{\pi i/3})|^2$ and $|F_{g}(1,-1)|$ are of positive type on $F$. In particular they are of positive type on $\vec{F}$.
\end{proposition}
Notice that in the previous statement we have considered $| V_g (e^{\pi i/3})|^2$ with $g \in F$ since, for $t$ a complex number of modulus one,
$|V_{\vec{L}}(t)|$ is actually independent of the orientation of the components of the link $L$ (see e.g. \cite[p.26]{Lic}), as well as $|F_{\vec{L}}(1,-1)|$.

\medskip

It is well-known that for an oriented link $\vec{L}$, $V_{\vec{L}}(t)$ coincides with $(-A^3)^{-\textrm{wr}(\vec{L})} \langle L \rangle (A)$ for $t^{-1/4} = A$. 
We proved in \cite{AiCo1} that the Kauffman bracket function is of positive type on $F$ when evaluated at certain roots of unity. 
In order to examine the issue whether the Jones polynomial gives rise to a function of positive type, one runs into the corresponding problem for the factor 
$(-A^3)^{-\textrm{wr}(\vec{L})}$, where $\textrm{wr}$ is the writhe number. 
At first sight, it may seem unlikely that this function is of positive type for all values of $A$. However, this function is trivially of positive type when $A=-1$ (see Remark \ref{tequal1}), so it is certainly meaningful to ask for which values of $A$ it is actually of positive type.
Given an element $g\in \vec{F}$, consider the family of links associated to $g$. All these links differ by the addition of distant unknots. The writhe computed on any of these links assumes always the same value. Therefore, we have a (well defined) writhe function $\textrm{wr}(g):= \textrm{wr}(\vec{\L}(g))$, where $\vec{\L}(g)$ denotes any oriented link associated to $g\in\vec{F}$, and it makes sense to consider the function $(-A^3)^{-\textrm{wr}(\vec{\L}(g))}$ on $\vec{F}$, such that
$$V_g(t) = (-A^3)^{-\textrm{wr}(g)} \langle g \rangle (A), \quad g \in \vec{F}, \ t^{-1/4} = A \  , $$
where the Kauffman bracket function is defined as\footnote{Note that in \cite{AiCo1} it is used a slightly different 
but harmless 
normalization for the Kauffman function.}  $\langle g \rangle (A) = \langle \L(g) \rangle (-A^2-A^{-2})^{-n+1}$.
We will also need the following facts, that will be used hereafter without further mention.
\begin{remark} \label{Jones-Laur}
We stress that the Jones polynomial takes values in the ring $\mathbb{Z}[t^{\pm 1/4}]$ and the symbol $t$ actually stands for $(t^{1/4})^4$. Moreover it can be proved that actually $V_{\vec{L}}(t)\in \mathbb{Z}[t^{\pm 1/2}]$, so when one wants to evaluate the Jones polynomial at a specific point, it is necessary to specify the value of $t^{1/2}\equiv A^{-2}$, not just of $t$. As a matter of fact, the Jones polynomial of an oriented link is a Laurent polynomial in $t$ if the number of components is odd, while it is the product of such a polynomial with $\sqrt t$ otherwise.
\end{remark}

We are now ready to show that the writhe function is a homomorphism ${\rm wr}: \vec{F}\to\mathbb{Z}$. Similarly to \cite{AiCo1}, given $g=g(T_+,T_-)\in \vec{F}$, the associated oriented link $\vec{\L}(g(T_+,T_-))$ may be decomposed in the two upper and lower halves, that is $\vec{\L}(g(T_+,T_-))=(\vec{\L}_+(T_+),\vec{\L}_-(T_-))$, hereafter called oriented \emph{semi-links}. 
We thus have ${\rm wr}(\vec{\L}(g))={\rm wr}(\vec{\L}_+(T_+))+{\rm wr}(\vec{\L}_-(T_-))$.

\begin{proposition}\label{wr-hom-F}
For any $g, h \in\vec{F}$, we have that ${\rm wr}(gh)={\rm wr}(g)+{\rm wr}(h)$.
\end{proposition}
\begin{proof}
We begin recalling the procedure used by Jones to endow the Seifert surface associated to $\vec{\L}(g)$ with a canonical orientation. Each element of $\vec{F}$ has associated $\Gamma$-graph that admits two $2$-colourings by the two colours $\pm$. By convention we choose the colouring such that the first vertex on the left has colour $+$. We  consider the shading of a link diagram  associated to $g$ (as usual the outer region is white).  
The surface can then be obtained by gluing the shaded regions with suitably twisted strips. Indeed any such region corresponds to a vertex of the
$\Gamma$-graph, and its orientation is determined by the colour assigned to the vertex, namely counterclockwise if the colour is $+$ and clockwise
in the other case. In particular, it follows that the orientation of the boundary strands of the upper semi-link is determined by the colouring via the rule
$+ =  \downuparrows$, $- = \updownarrows$. 

Therefore, if we now consider $g=g(T_+,T)$ and $h=g(T,T_-)$ then $gh = g(T_+,T_-)$  and it is easy to see that
\begin{eqnarray*}
{\rm wr}(\vec{\L}(g(T_+,T))) & = & {\rm wr}(\vec{\L}_+(T_+))+{\rm wr}(\vec{\L}_-(T))\\
{\rm wr}(\vec{\L}(g(T,T_-))) & = & {\rm wr}(\vec{\L}_+(T))+{\rm wr}(\vec{\L}_-(T_-))\\
& = & -{\rm wr}(\vec{\L}_-(T))+{\rm wr}(\vec{\L}_-(T_-))
\end{eqnarray*}
where the negative sign of the summand  in the last formula is due to the passage to the mirror image (cf. \cite[Prop. 2.4]{AiCo1}
; notice that the different orientation of the semilinks $\vec{\L}_+(T)$ and $\vec{\L}_-(T)$ does not affect the writhe number).
Thus we get 
$$\textrm{wr}(gh)= {\rm wr}(\vec{\L}_+(T_+))+{\rm wr}(\vec{\L}_-(T_-))=\textrm{wr}(g)+\textrm{wr}(h).$$
\end{proof}

\begin{example} \label{ese-x0x1}
As already mentioned, the Jones-Thompson group $\vec{F}$ is generated by three elements $x_0x_1$, $x_1x_2$ and $x_2x_3$. Here are a pair of trees representing $x_0x_1$, the associated $\Gamma$-graph and medial graph
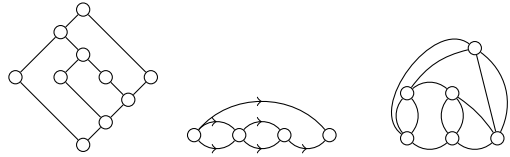
\[\begin{tikzpicture}[x=0.6cm, y=0.6cm,
    every edge/.style={
        draw,
      postaction={decorate,
                    decoration={markings}
                   }
        }
]
	\vertex (a) at (0,0) {};
	\vertex (b) at (1,0) {};
	\vertex (c) at (2,0) {};
	\vertex (d) at (3,0) {};
	\vertex (e) at (1.5,0.5) {};
	\vertex (f) at (1,1) {};
	\vertex (g) at (1.5,1.5) {};
	\vertex (h) at (2.5,-0.5) {};
	\vertex (i) at (2,-1) {};
	\vertex (l) at (1.5,-1.5) {};

	\path
		(a) edge (f)
		(b) edge (e)
		(c) edge (e)
		(e) edge (f)
		(g) edge (f)
		(d) edge (g)
		(a) edge (l)
		(b) edge (i)
		(c) edge (h)
		(d) edge (h)
		(h) edge (i)
		(i) edge (l);
\end{tikzpicture}
\quad
\begin{tikzpicture}[x=0.6cm, y=0.6cm,
    every edge/.style={
        draw,
        postaction={decorate,
                    decoration={markings,mark=at position 0.5 with {\arrow{>}}}
                   }
        }
]
	\vertex (a) at (0,0) {};
	\vertex (b) at (1,0) {};
	\vertex (c) at (2,0) {};
	\vertex (d) at (3,0) {};
	\path
		(a) edge[bend left=50] (b)
		(b) edge[bend left=50] (c)
		(a) edge[bend left=50] (d)
		(a) edge[bend right=50] (b)
		(b) edge[bend right=50] (c)
		(c) edge[bend right=50] (d)
;
\end{tikzpicture}
\quad
\begin{tikzpicture}[x=0.6cm, y=0.6cm,
    every edge/.style={
        draw,
      postaction={decorate,
                    decoration={markings}
                   }
        }
]
	\vertex (a) at (0,0) {};
	\vertex (b) at (1,0) {};
	\vertex (c) at (2,0) {};
	\vertex (d) at (0,1) {};
	\vertex (e) at (1,1) {};
	\vertex (f) at (1.50,2) {};
	\path
		(a) edge[bend left=40] (d)
		(a) edge[bend right=40] (d)
		(b) edge[bend left=40] (e)
		(b) edge[bend right=40] (e)
		(a) edge[bend right=40] (b)
		(b) edge[bend right=40] (c)
		(d) edge[bend left=40] (e)
		(e) edge[bend left=10] (c)
		(c) edge (f)
		(d) edge[bend left=20] (f)
		(c) edge[bend right=40] (f)
		(a) edge[bend right=270] (f)
;
\end{tikzpicture}
\]
The associated link diagram is shown below, representing a two-component trivial link.
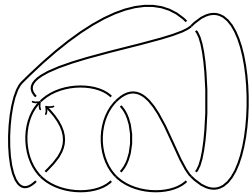
\[\begin{tikzpicture}[every path/.style={thick}, every node/.style={transform shape, knot crossing, inner sep=1.5pt}]
\node (a) at (0,0) {};
\node (b) at (1,0) {};
\node (c) at (2,0) {};
\node (d) at (0,1) {};
\node (e) at (1,1) {};
\node (f) at (2,2) {};
\node (x) at (-0.25,1.25) {};

\draw (a.center) .. controls (a.4 south east) and (b.4 south west) ..  (b);
\draw[->] (a.center) .. controls (a.4 north west) and (d.4 south west) ..  (d.center);
\draw[->] (a) .. controls (a.4 north east) and (d.8 south east) ..  (d);

\draw (b.center) .. controls (b.4 north west) and (e.4 south west) ..  (e.center);
\draw (b) .. controls (b.4 north east) and (e.4 south east) ..  (e);
\draw (b.center) .. controls (b.4 south east) and (c.4 south west) ..  (c);

\draw (c) .. controls (c.4 north east) and (f.4 south east) ..  (f);
\draw (c.center) .. controls (c.16 south east) and (f.16 north east) ..  (f.center);

\draw (d) .. controls (d.8 north west) and (f.4 south west) ..  (f.center);
\draw (d.center) .. controls (d.4 north east) and (e.4 north west) ..  (e);

\draw (e.center) .. controls (e.8 north east) and (c.4 north west) ..  (c.center);

\draw (a) .. controls (a.8 south west) and (x.4 south west) .. (x.center).. controls (x.16 north east) and (f.8 north west) ..  (f);

\end{tikzpicture}\] 
A quick computation shows that ${\rm wr}(\L(x_0x_1))=0$.
\end{example}

\begin{lemma} \label{shift}
Consider an element $g=g(T_+,T_-)\in F$, where $T_+, T_-$ have $n$ leaves and denote by $\phi: F\to F$ the shift homomorphism defined by
$\phi(x_i) = x_{i+1}$, $i \geq 0$. Then $\L(\phi(g(T_+,T_-)))=\L(g(T_+,T_-))\cup O$, where $O$ denotes a distant unknot. 
\end{lemma}
\begin{proof}
First of all we make some observation concerning the  pictorial interpretation of the map $\phi$. The following pairs of trees represent $g$ and $\phi(g)$, respectively
$$
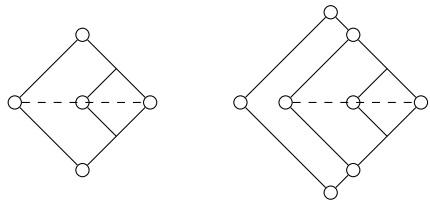
\begin{tikzpicture}[x=0.6cm, y=0.6cm,
    every edge/.style={
        draw,
      postaction={decorate,
                    decoration={markings}
                   }
        }
]
	\vertex (a) at (0,0) {};
	\vertex (b) at (3,0) {};
	\vertex (c) at (1.5,1.5) {};
	\vertex (d) at (1.5,-1.5) {};
	\vertex (e) at (1.5,0) {};
	\path
		(a) edge (c)
		(b) edge (c)
		(a) edge (d)
		(d) edge (b)
		(2.25,0.75)  edge (e)
		(2.25,-0.75) edge (e)
		(a) edge[dashed]  (b);
	\vertex (a1) at (6,0) {};
	\vertex (b1) at (9,0) {};
	\vertex (c1) at (7.5,1.5) {};
	\vertex (d1) at (7.5,-1.5) {};
	\vertex (e1) at (7.5,0) {};
	\vertex (h1) at (5,0) {};
	\vertex (i1) at (7,2) {};
	\vertex (l1) at (7,-2) {};


	\path
		(a1) edge (c1)
		(b1) edge (c1)
		(a1) edge (d1)
		(d1) edge (b1)
		(8.25,0.75)  edge (e1)
		(8.25,-0.75) edge (e1)
		(h1) edge (i1)
		(h1) edge (l1)
		(d1) edge (l1)
		(c1) edge (i1)
		(a1) edge[dashed]  (b1);
\end{tikzpicture}
$$
Thus, considering the associated links, the outputs differ by the addition of a distant unknot.
\end{proof}
There is an oriented version of the former lemma, namely $\phi$ maps $\vec{F}$ into $\vec{F}$ and, if $g \in \vec{F}$, then the oriented link associated to $\phi(g)$ consists of the link associated to $g$ with opposite orientation and an unknot. Recall that the Jones polynomial is invariant under a simultaneous change of orientation of all the components. Thus we immediately get that $V_g(t)=V_{\phi(g)}(t)$
\begin{proposition}\label{wrzero}
The writhe function is identically zero on $\vec{F}$.
\end{proposition}
\begin{proof}
As $\phi(x_i)=x_{i+1}$,  Lemma \ref{shift} and Example \ref{ese-x0x1} show that $\rm{wr}(x_1x_2)=\rm{wr}(x_2x_3)=0$. As these elements generate $\vec{F}$ we get the thesis.
\end{proof}

\begin{remark}\label{tequal1}
For $A=-1$, $(-A^3)^{-\textrm{wr}(\vec{L})} = 1$ for any $\vec{L}$, thus $(-A^3)^{-\textrm{wr}(\vec{\L}(g))}  = 1$ for all $g \in \vec{F}$ 
(obviously a function of positive type on $\vec{F}$).
Hence, $t = 1$ and, recalling that $V_{\vec{L}}(1) = (-2)^{c(L)-1}$ where $c(L)$ is the number of components of the link $L$ \cite[Theorem 15]{JovN},
we get $V_g(1) = \langle g \rangle(-1) = (-2)^{c(\L(T_+,T_-))-n}$ where $g=g(T_+,T_-)$ for some trees $T_+, T_-$ with $n$ leaves.
It is thus natural to ask whether $V_g(t)$ is of positive type on $\vec{F}$ for $t=1$, or even if $\langle g \rangle (A)$ is of positive type on $F$ for $A=-1$ (they coincide on $\vec{F}$). A positive answer to the first question will be given in the sequel.
\end{remark}

\begin{remark}
It is not difficult to show that 
$$
V_{g^{-1}}(t) = V_g(1/t) , 
$$
since, for $g=g(T_+,T_-)$ where $T_\pm$ have $n$ leaves, 
$$V_{g^{-1}}(t) = \frac{V_{\vec{\L}(T_-,T_+)}(t)}{(-t^{1/2} - t^{-1/2})^{n-1}} 
= \frac{V_{\stackrel{\shortleftarrow}{\L}(T_+,T_-)^*}(t)}{(-t^{1/2} - t^{-1/2})^{n-1}} 
= \frac{V_{\vec{\L}(T_+,T_-)}(t^{-1})}{(-t^{1/2} - t^{-1/2})^{n-1}} 
= V_g(t^{-1}) 
$$
where $\stackrel{\shortleftarrow}{L}$ denotes $\vec{L}$ with the opposite orientation,
$L^*$ is the mirror image of $L$
and we also used that $\L(T_-,T_+)=\L(T_+,T_-)^*$ \cite[Prop.2.4]{AiCo1} and \cite[Theor.3]{JovN}.
\end{remark}

In the sequel we will make use of the following result, which has been obtained in \cite{AiCo1}.
\begin{theorem}
Let $Q \in \{2, 3, 4\}$ 
and let $A$ be any solution of the equation $$A^2+A^{-2}+\sqrt{Q}=0 \ , $$  namely 
$A\in \{\pm e^{\frac{3\pi i}{8}}, \pm e^{-\frac{3 \pi i}{8}}, \pm e^{\frac{5\pi i}{12}}, \pm e^{-\frac{5 \pi i}{12}}, \pm i \}$.
Then the Kauffman bracket function $\langle g \rangle (A)$ is of positive type on $F$.
\end{theorem}
For several values of the variable $t^{1/2}$, the Jones polynomial has topological interpretations (see e.g. \cite{Oht}, p.383). For example
\begin{eqnarray}
V_{\vec{L}} (t=1)& = & (-2)^{c(L)-1}, \label{val-jones-1} \\
V_{\vec{L}} (t=i) &=& \left\{\begin{array}{ll} (-\sqrt{2})^{c(L)-1}(-1)^{{\rm Arf}(L)} & \textrm{ if Arf($L$) exists} \label{val-jones-2}\\
0 & \textrm{otherwise ,}
 		\end{array}
 		\right.\\
V_{\vec{L}} (t=e^{\frac{i \pi}{3}}) &=& \pm i^{c(L)-1}\sqrt{-3}^{\; {\rm dim}(H_1(M_{2,L}, \mathbb{Z}_3))} \label{val-jones-3},
\end{eqnarray}		 
where $M_{2,L}$ denotes the double branched cover of $S^3$ branched along $L$.

We now give an interpretation for $t^{1/2}=-1$. 
\begin{proposition}\label{sqrt-1}
The following equality holds
$$
V_{\vec{L}} (t^{1/2} =-1)=2^{c(L)-1} = (-1)^{c(L)-1} V_{\vec{L}} (1).
$$
\end{proposition}
\begin{proof}
We essentially follow the proof of the case $t^{1/2}=1$ (cf. Proposition 11.2.6, p. 228, \cite{Mura}). When $t^{1/2}=-1$, the skein relation implies that 
$$V_\overcrossing(t^{1/2} =-1)=V_\undercrossing(t^{1/2} =-1)=V_{O\cdots O} (t^{1/2} =-1).$$
By Proposition 11.1.1, p. 220, \cite{Mura}, we have that $V_{O\cdots O} (t^{1/2})=(-t^{1/2}-t^{-1/2} )^{c(L)-1}$ and the claim follows.
\end{proof}

Given $g=g(T_+,T_-)\in F$, where $T_+, T_-$ are trees with $n$ leaves, it is convenient to consider the integer-valued 
function $c(g):=c(\L(T_+,T_-))-n$ on $F$.
Notice that $c(e)=0$, i.e. $c(\L(T,T)) = n$ for every rooted tree $T$ with $n$ leaves, $c(g)=c(g^{-1})$ and $c(g) \leq 0$ for all $g \in F$.

\begin{proposition}\label{lemma-comp-link-F}
We have that ${c(g)} \in -2{\mathbb N}_0$ for all $g \in \vec{F}$.
\end{proposition}
\begin{proof}
With the above notations denote by $\vec{\L}$ the link $\vec{\L}(T_+,T_-)$. Let $b(\vec{L})$ be the number of black regions in the shaded diagram of $\vec{L}$. Since we are only interested in the number of the components of the link, after changing some crossings whenever necessary we get a new link $\vec{L'}$ which is trivial. We observe that $b(\vec{\L})=b(\vec{L'})$.  Denote by ${\rm r}(\vec{\L})$ the rotation number of an oriented link diagram (see \cite{Jae} and the references therein). We have that
\begin{align*}
(-1)^{c(g)} & = (-1)^{c(\vec{\L})-n}=(-1)^{c(\vec{\L})}(-1)^{n}=\\
 & = (-1)^{c(\vec{L'})}(-1)^{n}=(-1)^{{\rm r}(\vec{L'})}(-1)^{n}=\\
 & = (-1)^{b(\vec{\L})}(-1)^{n}= (-1)^{n}(-1)^{n}= 1\\
\end{align*}
where we used the fact that the rotation number may be computed resolving each crossing $\overcrossing$ and $\undercrossing$ as $\upuparrows$, and that the number of black regions in the shaded diagram is equal to the number of the vertices in the face graph.
\end{proof}

\smallskip
After this preparation, we are now ready to state the main results of this section.

\begin{proposition}\label{Vg1}
The function $V_g(1) = 2^{c(\L(T_+,T_-))-n} $ is of positive type on $\vec{F}$  for $g=g(T_+, T_-)$, where $n$ is the number of leaves of $T_+$ and $T_-$.
\end{proposition}
\begin{proof}
Since the Kauffman function is of positive type (on $F$) for $A=\pm i$ and, by Proposition \ref{wrzero}, ${\rm wr}\equiv 0$, we have that
\begin{align*}
	\langle g\rangle (\pm i) & = \langle \L(T_+, T_-)\rangle (A) (-A^2-A^{-2})^{-n+1}|_{A=\pm i}\\
	& = (-A^{-3})^{{\rm wr}(\vec{\L}(T_+,T_-))} \langle \L(T_+, T_-)\rangle (A) (-A^2-A^{-2})^{-n+1}|_{A= \pm i}\\
	& = (-1)^{c(\L(T_+,T_-))-1}V_{\vec{\L}(T_+,T_-)}(1) 2^{-n+1}\\
	& = (-1)^{c(\L(T_+,T_-))-1}V_{\vec{\L}(T_+,T_-)}(1) (-2)^{-n+1} (-2)^{n-1} 2^{-n+1}\\
	& = (-1)^{c(\L(T_+,T_-))-n}V_g(1)\\
	& = V_g(1) \\
         & = (-2)^{c(\L(T_+,T_-))-1}(-2)^{-n+1}\\
	& = (-1)^{c(\L(T_+,T_-))-n} 2^{c(\L(T_+,T_-))-n} \\
	& = 2^{c(\L(T_+,T_-))-n} \\
\end{align*}
where we have used the relation $V_{\vec{\L}(T_+,T_-)}(1)=(-1)^{c(\L(T_+,T_-))-1}\langle \L(T_+,T_-)\rangle (\pm i)$ in the third equality 
(see Remark \ref{Jones-Laur}, cf. Proposition \ref{sqrt-1}), Proposition \ref{lemma-comp-link-F} in the sixth and the ninth equalities and the identity \eqref{val-jones-1} in the seventh equality.
\end{proof}

\begin{proposition}
The function $V_g(i)$ is of positive type on $\vec{F}$. 
\end{proposition}
\begin{proof}
Since the Kauffman function is of positive type (on $F$) for $A=\pm e^{\frac{3\pi i}{8}}$ and, by Proposition \ref{wrzero}, ${\rm wr}\equiv 0$, we have that
\begin{align*}
	\langle g\rangle (\pm e^{\frac{3\pi i}{8}}) & = \langle \L(T_+, T_-)\rangle (A) (-A^2-A^{-2})^{-n+1}|_{A=\pm e^{\frac{3\pi i}{8}}}\\
	& = (-A^{-3})^{{\rm wr}(\vec{\L}(T_+,T_-))} \langle \L(T_+, T_-)\rangle (A) (-A^2-A^{-2})^{-n+1}|_{A= \pm e^{\frac{3\pi i}{8}}}\\
	& = (-1)^{c(\L(T_+,T_-))-1}V_{\vec{\L}(T_+,T_-)}(i) (\sqrt{2})^{-n+1}\\
	& = (-1)^{c(\L(T_+,T_-))-1}V_{\vec{\L}(T_+,T_-)}(i) (-\sqrt{2})^{-n+1} (-\sqrt{2})^{n-1} (\sqrt{2})^{-n+1}\\
	& = (-1)^{c(\L(T_+,T_-))-n}V_g(i) \\
	& = V_g(i) 
\end{align*}
where we have used that $V_{\vec{\L}(T_+,T_-)}(i)=(-1)^{c(\L(T_+,T_-))-1}\langle \L(T_+,T_-)\rangle (\pm e^{\frac{3\pi i}{8}})$ (see Remark \ref{Jones-Laur}) and Proposition \ref{lemma-comp-link-F}.
\end{proof}
Since the Kauffman function is of positive type (on $F$) for $A=\pm e^{-\frac{3\pi i}{8}}$, 
a similar argument would yield that $V_g(-i)$ is of positive type too. However, 
$V_g(i) = V_g(-i)$ and thus we do not get any additional information.

\begin{proposition}\label{prop-V3}
The function $V_g(e^{\frac{\pi i}{3}})$ is of positive type on $\vec{F}$, as well as $V_g(e^{-\frac{\pi i}{3}})$.
\end{proposition}
\begin{proof}
Since the Kauffman function is of positive type (on $F$) for $A=\pm e^{\frac{5\pi i}{12}}$ and, by Proposition \ref{wrzero}, ${\rm wr}\equiv 0$, we have that
\begin{align*}
	\langle g\rangle (\pm e^{\frac{5\pi i}{12}}) & = \langle \L(T_+, T_-)\rangle (A) (-A^2-A^{-2})^{-n+1}|_{A=\pm e^{\frac{5\pi i}{12}}}\\
	& = (-A^{-3})^{{\rm wr}(\vec{L}(T_+,T_-))} \langle \L(T_+, T_-)\rangle (A) (-A^2-A^{-2})^{-n+1}|_{A= \pm e^{\frac{5\pi i}{12}}}\\
	& = (-1)^{c(\L(T_+,T_-))-1}V_{\vec{\L}(T_+,T_-)}(e^{-\frac{\pi i}{3}}) (\sqrt{3})^{-n+1}\\
	& = (-1)^{c(\L(T_+,T_-))-1}V_{\vec{\L}(T_+,T_-)}(e^{-\frac{\pi i}{3}}) (-\sqrt{3})^{-n+1} (-\sqrt{3})^{n-1} (\sqrt{3})^{-n+1}\\
	& = (-1)^{c(\L(T_+,T_-))-n}V_g(e^{-\frac{\pi i}{3}}) \\
	& = V_g(e^{-\frac{\pi i}{3}}) 
\end{align*}
where we have used that $V_{\vec{\L}(T_+,T_-)}(e^{-\frac{\pi i}{3}})=(-1)^{c(\L(T_+,T_-))-1}\langle \L(T_+,T_-)\rangle (\pm e^{\frac{5\pi i}{12}})$ (see Remark \ref{Jones-Laur}) and Proposition \ref{lemma-comp-link-F}.

Since the Kauffman function is of positive type (on $F$) for $A=\pm e^{-\frac{5\pi i}{12}}$ 
the second claim follows by a similar argument or, perhaps more directly, by noticing that $V_g(e^{-\frac{\pi i}{3}})$ is the complex conjugate
of $V_g(e^{\frac{\pi i}{3}})$.
\end{proof}
From Proposition \ref{prop-V3} one can recover the fact, already stated in Proposition \ref{threecol}, that $|V_g(e^{i \pi/3})|^2$ is positive definite on $\vec{F}$.

\medskip
Recall from \cite[(12.4), p. 369]{JoHecke} that, for $t=e^{2\pi i/3}$, one has $V_{\vec{L}}(t)=(-1)^{c(L)-1}$. Since
\begin{align*}
-t^{-1/2}-t^{1/2} & = - (e^{2\pi i/6}+e^{-2\pi i/6}) = -2\cos(\pi/3)=-1
\end{align*}
it holds, 
for $g=g(T_+,T_-)\in \vec{F}$, where $T_\pm$ have $n$ leaves,
$$
V_g(e^{2\pi i/3}) = (-1)^{c(\L(T_+,T_-))-1}(-1)^{-n+1}=(-1)^{c(g)}=1  
$$
by Proposition \ref{lemma-comp-link-F}.

\smallskip
Combining the previous results, we have thus proved the following statement.
\begin{theorem} \label{teo-fin-Jones}
The evaluations of the function $V_g(t)$ at $t = 1, e^{2\pi i/3}, i, e^{\pm i \pi/3}$ are of positive type on $\vec{F}$.
\end{theorem}
By different and more powerful methods, one can indeed show that the Jones function $V_g(t)$ is of positive type on $\vec{F}$ for $t \in \{e^{\pm 2\pi i/k} \ | \ k=3,4,5,\ldots\}$, see \cite{ACJ}.

\section{On the $2$-variable Kauffman polynomial as a function on $\vec{F}$}
The aim of this section is to prove that certain specialisations of the $2$-variable Kauffman polynomial give rise to functions of positive type on the Jones-Thompson group $\vec{F}$. We start recalling some results from \cite{Lip} (also discussed in \cite{DLH-J}). 

Let $\vec{L}$ be an oriented link. Its link diagram can be seen as a graph whose vertices are the crossings. We observe that every vertex is $4$-valent. For each vertex $x$, we denote by $(a_1^x,a_2^x)$ the upper string and by 
 $(a_1^{-x},a_2^{-x})$ the lower string. We assume that the four edges $a_1^{x}$, $a_1^{-x}$, $a_2^{x}$ and $a_2^{-x}$ appear in anticlockwise order. We define the following function $w: \mathbb{Z}_2\times \mathbb{Z}_2\times \mathbb{Z}_2\times \mathbb{Z}_2\to \mathbb{C}$
\begin{eqnarray*}
		  w(s^+,s^-,t^+,t^-):= \left\{
 		\begin{array}{ll}
 		C & \textrm{ if $s^+=s^-\neq t^+=t^-$}  \\
 		C^{-1} & \textrm{ if $s^+=t^-\neq t^+=s^-$}  \\
 		0 & \textrm{ otherwise }  \\
 		\end{array}
 		\right.	
\end{eqnarray*}
Consider the partition function defined by 
$$
	Z_{L}(C) :=\sum_\tau \prod_{x\in V(\vec{L})} w(\tau(a_1^x),\tau(a_1^{-x}),\tau(a_2^{x}),\tau(a_2^{-x})) \ ,
$$
where the sum over the functions $\tau: E(\vec{L})\to \mathbb{Z}_2$, called \emph{states} of $\vec{L}$.  
\begin{theorem} (\cite{Lip}, Theorem 1, p.339)\label{teo-2-var-kauff}
For any $C\in \mathbb{C}$, the following equality holds
$$
	\frac{C^{{\rm wr} (\vec{L})} (-1)^{c(L)-1}}{2} Z_{L}(C) = F_{\vec{L}}(iC^{-1}, iC-iC^{-1}).
$$
\end{theorem}

We are now ready to state and prove the main result of this section.
\begin{proposition}\label{2k}
For any $C\in \mathbb{R}\setminus\{0, 1\} $
the 
function $F_g(iC^{-1}, iC-iC^{-1})$ is of positive type on $\vec{F}$.
\end{proposition}
\begin{proof}
Without loss of generality we can suppose that $g_i=g(T_+^i,T_-)$  where $T_+^i$, $i=1, \cdots, r$ and $T_-$ have $n$ leaves  and $g_ig_j^{-1}=g(T^i_+,T^j_+)$. A simple computation shows that, for $a=iC^{-1}$ and $x=iC-iC^{-1}$,  the following equality holds 
$$
(a^{-1}+a)x^{-1}-1=-2.
$$
By Proposition \ref{wrzero} we may omit the factor $C^{{\rm wr} (\vec{L})}/2$. 
We have to prove that the matrix $\Big(F_{\vec{\L}(T^i_+,T^j_+)}(iC^{-1}, iC-iC^{-1})/(-2)^{n-1}\Big)_{i,j=1}^r$ is positive semi-definite. We have that
\begin{align*}
F_{\vec{\L}(T^i_+,T^j_+)}(iC^{-1}, iC-iC^{-1})/(-2)^{n-1} 
& = \frac{(-1)^{c(\L(T^i_+,T^j_+))-1} Z_{\L(T^i_+,T^j_+)}(C)}{2 (-2)^{n-1}}  \\
& = \frac{(-1)^{c(\L(T^i_+,T^j_+))-1} Z_{\L(T^i_+,T^j_+)}(C)}{2 (-1)^{n-1} 2^{n-1}} \\
& = \frac{(-1)^{c(\L(T^i_+,T^j_+))-n} Z_{\L(T^i_+,T^j_+)}(C)}{2^{n}} \\
& =  \frac{Z_{\L(T^i_+,T^j_+)}(C)}{2^{n}}
\end{align*}
by Proposition \ref{lemma-comp-link-F}.
Since the factor $1/2^n$ is positive, we may neglect it. Thus we have to prove that the matrix $(Z_{\vec{\L}(T_+^i,T_+^j)}(C))_{i,j=1}^r$ is positive semidefinite. 

We have that
\begin{eqnarray*}
	Z_{\L(T_+,T_-)}(C) &=& \sum_\tau \prod_{x\in V(\vec{\L}(T_+,T_+))} w(\tau(a_1^x),\tau(a_1^{-x}),\tau(a_2^{x}),\tau(a_2^{-x})) \\
	&=& \sum_{\tau_0} \left(\sum_{\tau_+} \prod_{x\in V(\vec{\L}(T_+))} w(\tau(a_1^x),\tau(a_1^{-x}),\tau(a_2^{x}),\tau(a_2^{-x}))\right) \\
	& & \times   \left(\sum_{\tau_-} \prod_{x\in V(\vec{\L}(T_-))} w(\tau(a_1^x),\tau(a_1^{-x}),\tau(a_2^{x}),\tau(a_2^{-x})) \right)\\
\end{eqnarray*}
where we have decomposed each state $\tau$ as $(\tau_0,\tau_+,\tau_-)$, $\tau_0$ being a function on edges in common between the semi-links, $\tau_+$ a function on the the remaining edges in $\vec{\L}(T_+)$ and $\tau_-$ a function on the the remaining edges in $\vec{\L}(T_-)$.
For any $\tau_0=(\tau_1,\cdots ,\tau_{2n})$, the expression $\sum_{\tau_+} \prod_{x\in V(\vec{\L}(T_+))} w(\tau(a_1^x),\tau(a_1^{-x}),\tau(a_2^{x}),\tau(a_2^{-x}))$ defines the $\tau$-th component of a vector $v_{T_+}$ in $\H=\mathbb{C}^{2^{2n}}$. The choice of the vector space $\mathbb{C}^{2}$ is due to the two different values assigned to each edge. Thus, we may define  vectors $v_{T_+^i}$ for all $i=1,\ldots , r$ such that 
$$
	Z_{\vec{\L}(T_+, T_-)}=\langle v_{T_+^i},v_{T_+^j}\rangle.
$$
It follows 
that the matrix $\Big(F_{\vec{\L}(T^i_+,T^j_+)}(iC^{-1}, iC-iC^{-1})/(-2)^{n-1}\Big)_{i,j=1}^r$ is positive semi-definite for any $r$, i.e. the function $F_g(iC^{-1}, iC-iC^{-1})$ is of positive type.
\end{proof}

Looking back at Theorem \ref{teo-fin-Jones}, it is straightforward to observe the following fact.
\begin{corollary}
The evaluations of the $2$-variable Kauffman functions 
$$F_g(-1,2), \quad F_g(-e^{-i 3 \pi/8},2\cos(\pi/8)), \quad F_g(-e^{-i  \pi/4},2\cos(\pi/12))$$ 
are of positive type on $\vec{F}$.
\end{corollary}
\begin{proof}
The statement follows at once from Theorem \ref{teo-fin-Jones} and the following relation (see e.g. \cite[Prop. 16.6]{Lic}) between the Jones polynomial and the $2$-variable Kauffman polynomial 
$$
V_{\vec{L}}(t)=F_{\vec{L}}(-t^{-3/4},t^{1/4}+t^{-1/4}).
$$
\end{proof}
However, a direct comparison reveals that only the first of these evaluations can be reobtained by means of Proposition \ref{2k}, when $C=i$.
All the other evaluations of the 2-variable Kauffman function in Proposition \ref{2k} do not refer to the Jones function.

\section{On the Homfly polynomial as a function on $\vec{F}$}
In this section we examine whether certain specialisations of the Homfly polynomial $P$ give rise to functions of positive type on the Jones-Thompson group $\vec{F}$. We briefly recall a vertex model whose associated partition function is related to the Homfly polynomial \cite{JoStat,Tur}. For a detailed exposition we refer to \cite[Theorem 3.1]{Jae}, cf. \cite[Chapter 8]{Kawa}. 

Let $\vec{L}$ be an oriented link diagram. 
For $k\in\mathbb{N}$, the elements in $\Theta_k:=\{1,\ldots,k\}$ are called colours. The functions $\tau: E(\vec{L})\to \Theta_k$ are called \emph{states}.
Consider a $4$-valent vertex with colours $i$ and $j$ as inputs, $h$ and $l$ as outputs. We define the following weights 
$w_\pm: \Theta_k\times \Theta_k\times \Theta_k\times \Theta_k\to \mathbb{C}$
for positive and negative crossings, respectively,
as
\begin{eqnarray*}
		  w_+(i,j;h,l):= \left\{
 		\begin{array}{ll}
 		q-q^{-1} & \textrm{ if $i<j$, $i=h$, $j=l$}  \\
 		1 & \textrm{ if $i=l$, $j=h$, $i\neq j$ }  \\
 		q & \textrm{ if $i=j=h=l$ }  \\
 		0 & \textrm{ otherwise }  \\
 		\end{array}
 		\right.	\\
		w_-(i,j;h,l):= \left\{
 		\begin{array}{ll}
 		q^{-1}-q & \textrm{ if $i>j$, $i=h$, $j=l$}  \\
 		1 & \textrm{ if $i=l$, $j=h$, $i\neq j$ }  \\
 		q^{-1} & \textrm{ if $i=j=h=l$ }  \\
 		0 & \textrm{ otherwise }  \\
 		\end{array}
 		\right.	\\
\end{eqnarray*}

Consider the partition function given by  
$$
	Z_{\vec{L}}(q,k) :=\frac{q-q^{-1}}{q^k-q^{-k}} q^{-k{\rm wr}(\vec{L})} q^{-(k+1)r(\vec{L})}\sum_\tau \prod_{x\in V(\vec{L})} w(\cdot) q^{2s(\vec{L},\tau)} \ ,
$$
where the sum runs over the state functions, $w(\cdot)$ denotes the appropriate weight function, namely $w_\pm$ if the vertex is a positive/negative crossing, ${\rm r}(\vec{L})$ is the rotation number of $\vec{L}$ and $s(\vec{L},\tau)$ is a suitable integer depending on $\tau$ and the rotation numbers of some other diagrams associated to $\vec{L}$.
It can be proved that the above partition function satisfies the skein relation
$$
q^k Z_{\overcrossing}(q,k)-q^{-k} Z_{\undercrossing}(q,k)=(q-q^{-1})Z_{\upuparrows}(q,k)
$$
and thus coincides with the Homfly polynomial $P_{\vec{L}}(q^k,q-q^{-1})$. A simple computation shows that, for $\alpha=q^k$ and $z=q-q^{-1}$ as above,  the following equality holds 
\begin{eqnarray*}
\frac{(\alpha-\alpha^{-1})}{z} & = & \frac{q^k-q^{-k}}{q-q^{-1}}\\
& = & \frac{q^{-k}(q^{2k}-1)}{q-q^{-1}}\\
& = & \frac{q^{-k}(q^2-1)(q^{2(k-1)}+q^{2(k-2)}+\ldots +1)}{q-q^{-1}}\\
& = & q^{-k+1} (q^{2(k-1)}+q^{2(k-2)}+\ldots +1).
\end{eqnarray*}

\begin{lemma}
For $q=1$, it holds $\tilde{Z}_{\vec{L}}(1,k) := \sum_\tau \prod_{x\in V(\vec{L})} w(\cdot) = k^{c(\vec{L})}$.
\end{lemma}
This is an easy consequence of the fact that, due to the particular form of the weights, the only contributions to the above sum arise from the states that are constant on each component of the link diagram.

A related observation is that there exists the limit 
$$\lim_{q\to  1} Z_{\vec{L}}(q,k) = k^{c(\vec{L})-1} \ . $$
Although for $q \neq 1$ we do not get  a positive definite function, we will recover such a function through this limit (after proper renormalization).

We are now ready to state and prove the main theorem of this section, that extends Proposition \ref{Vg1}.
\begin{theorem}\label{homfly}
For each nonzero integer $k$, the function $k^{c(g)}$ is of positive type on $\vec{F}$.
\end{theorem}
\begin{proof}
Suppose that $k\in \mathbb{N}$.
Without loss of generality we can suppose that $g_i=g(T_+^i,T_-)$  where $T_+^i$, $i=1, \cdots, r$ and $T_-$ have $n$ leaves  and $g_ig_j^{-1}=g(T^i_+,T^j_+)$. 
We may neglect the positive factor $1/k^n$ and prove that the matrix $\Big(k^{c(\vec{\L}(T_+^i,T_+^j))}\Big)_{i,j=1}^r$ is positive semi-definite. 

\noindent But the latter matrix coincides with the matrix $\Big(\tilde{Z}_{\vec{\L}(T_+^i,T_+^j)}(1,k)\Big)_{i,j=1}^r$. 
Now we have that
\begin{eqnarray*}
	\tilde{Z}_{\vec{\L}(T_+,T_-)}(1,k) & = & \sum_\tau \prod_{x\in V(\vec{\L}(T_+,T_-))} w(\cdot) \\
&=& \sum_{\tau_0} \left(\sum_{\tau_+} \prod_{x\in V(\vec{\L}(T_+))} w(\cdot)\right) 
	 \left(\sum_{\tau_-} \prod_{x\in V(\vec{\L}(T_-))} w(\cdot) \right)\\
\end{eqnarray*}
where we have decomposed each state $\tau$ as $(\tau_0,\tau_+,\tau_-)$, $\tau_0$ being a function on edges in common between the semi-links, $\tau_+$ a function on the the remaining edges in $\vec{\L}(T_+)$ and $\tau_-$ a function on the the remaining edges in $\vec{\L}(T_-)$.
For any $\tau_0=(\tau_1,\cdots ,\tau_{2n})$, the expression $\sum_{\tau_+} \prod_{x\in V(\vec{\L}(T_+))} w(\cdot)$ defines the $\tau_0$-th component of a vector $v_{T_+}$ in $\H=\mathbb{C}^{k^{2n}}$. 
In view of the specific form of the weights, 
it is not difficult to check that
$$
	\tilde{Z}_{\vec{\L}(T_+, T_-)}(1,k)=\langle v_{T_+^i},v_{T_+^j}\rangle .
$$
Using Proposition \ref{lemma-comp-link-F} we are done.
\end{proof}

By comparison with the corresponding result about the Jones function (Theorem \ref{teo-fin-Jones}) we also obtain:
\begin{corollary}
The 
evaluations of the 
Homfly functions $P_g(1,0)$, $P_g(-i,2i\sin(\pi/4))$ and $P_g(e^{-i  \pi/3},2i\sin(\pi/6))$ are of positive type 
on $\vec{F}$.
\end{corollary}
\begin{proof}
The statement follows at once from Theorem \ref{teo-fin-Jones} and the following relation between the Jones polynomial and the Homfly polynomial (\cite{F}, p.240)
$$
V_{\vec{L}}(t)=P_{\vec{L}}(t^{-1},t^{1/2}-t^{-1/2}) \ , 
$$
which implies the equality $V_g(t) = P_g(t^{-1},t^{1/2}-t^{-1/2})$.
\end{proof}

\section{Functions of positive type on Thompson group $T$}

We discuss the natural generalization of some results in \cite{AiCo1} to the Thompson group $T$, which is the group of piecewise linear homeomorphisms of the circle (the interval $[0,1]$ with identified endpoints) that are differentiable outside a finite set of rational dyadics and with slopes in $2^{\mathbb Z}$. We are going to prove that the statements about the Kauffman bracket and the Tutte polynomial obtained in \cite{AiCo1} keep their validity also for $T$.

An element of $T$ can be described by a pair of rooted binary trees $(T_+,T_-)$ with the same number of leaves (as for $F$) and such that, in addition, one leaf of $T_-$ has a mark, meaning that it can be joined to the first leaf of $T_+$ and then coherently identifying the remaining leaves in cyclic order.
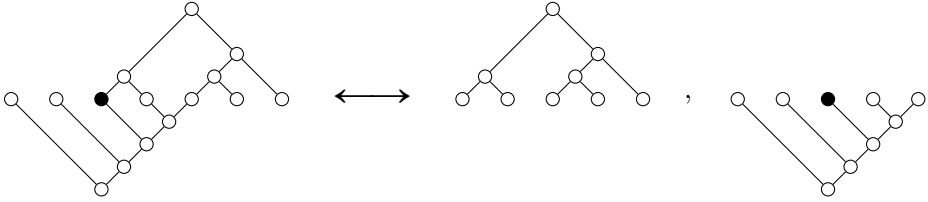
\[
\begin{tikzpicture}[x=0.6cm, y=0.6cm,
    every edge/.style={
        draw,
      postaction={decorate,
                    decoration={markings}
                   }
        }
]
	\vertex (a) at (0,0) {};
	\vertex (b) at (1,0) {};
	\vertice (c) at (2,0) {};
	\vertex (d) at (3,0) {};
	\vertex (e) at (4,0) {};
	\vertex (f) at (5,0) {};
	\vertex (g) at (6,0) {};
	\vertex (h) at (2.5,0.5) {};
	\vertex (i) at (4.5,0.5) {};
	\vertex (l) at (5,1) {};
	\vertex (m) at (4,2) {};
	\vertex (n) at (3.5,-0.5) {};
	\vertex (o) at (3,-1) {};
	\vertex (p) at (2.5,-1.5) {};
	\vertex (q) at (2,-2) {};

	\path
		(a) edge (q)
		(b) edge (p)
		(c) edge (o)
		(d) edge (n)
		(e) edge (n)
		(n) edge (o)
		(o) edge (p)
		(p) edge (q)
		(c) edge (h)
		(d) edge (h)
		(h) edge (m)
		(e) edge (i)
		(f) edge (i)
		(i) edge (l)
		(g) edge (l)
		(m) edge (l);
	\node [fill=none, scale=1.5] at (8,0) (node) {$\longleftrightarrow$};

	\vertex (c1) at (10,0) {};
	\vertex (d1) at (11,0) {};
	\vertex (e1) at (12,0) {};
	\vertex (f1) at (13,0) {};
	\vertex (g1) at (14,0) {};
	\vertex (h1) at (10.5,0.5) {};
	\vertex (i1) at (12.5,0.5) {};
	\vertex (l1) at (13,1) {};
	\vertex (m1) at (12,2) {};
	\node [fill=none, scale=1] at (15,0) (node) {,};

	\path
		(c1) edge (h1)
		(d1) edge (h1)
		(h1) edge (m1)
		(e1) edge (i1)
		(f1) edge (i1)
		(i1) edge (l1)
		(g1) edge (l1)
		(m1) edge (l1);
\end{tikzpicture}
\quad
\begin{tikzpicture}[x=0.6cm, y=0.6cm,
    every edge/.style={
        draw,
      postaction={decorate,
                    decoration={markings}
                   }
        }
]
	\vertex (a2) at (16,0) {};
	\vertex (b2) at (17,0) {};
	\vertice (c2) at (18,0) {};
	\vertex (d2) at (19,0) {};
	\vertex (e2) at (20,0) {};
	\vertex (n2) at (19.5,-0.5) {};
	\vertex (o2) at (19,-1) {};
	\vertex (p2) at (18.5,-1.5) {};
	\vertex (q2) at (18,-2) {};

	\path
		(a2) edge (q2)
		(b2) edge (p2)
		(c2) edge (o2)
		(d2) edge (n2)
		(e2) edge (n2)
		(n2) edge (o2)
		(o2) edge (p2)
		(p2) edge (q2);
\end{tikzpicture}
\]

Whenever convenient, we will use use the notation $(T_+,T_-,c_-)$, where $T_+$, $T_-$ have $n$ leaves, $c_- \in \{1,\ldots,n\}$ and the mark is on the $c_-$-th leaf of $T_-$.

Concerning the group structure, the inverse of an element $(T_+,T_-,c_-)$ as above is described by the pair $(T_-,T_+, n-c_- +2)$ (for the elements of $F$, i.e. when $c_-=1$, we tacitly assume that $n-c_- +2=1$ is understood mod $n$). 

As for $F$, the representation of an element in $T$ by (marked) trees as above is not unique, the freedom being that one can add or delete opposing carets. Multiplication of two elements $g,g'$ in $T$ can thus be performed by first choosing representatives $g=g(T_+,T_-,c_-)$ and $g'=g(T_-,T'_-,c'_-)$ so to obtain for $gg'$ the representative $g(T_+,T'_-,c_-+c'_- -1)$ (the last entry is defined mod $n$).

To an element $g=g(T_+,T_-,c_-)$ in $T$ we can associate a signed graph $\Gamma(T_+,T_-,c_-)$ as described in \cite{Jo}. This graph
is obtained by glueing an upper graph $\Gamma_+(T_+,c_-)$  and a lower graph $\Gamma_-(T_-)$ (independent of $c_-$).
As for $F$, by adding an opposing pair of carets, the $\Gamma$-graph changes by inserting a new vertex only connected to a vertex of $\Gamma(T_+,T_-,c_-)$ on its left by two parallel edges.

As usual, for any finite graph $G$ we consider the Tutte polynomial $T_G(x,y)$ in two variables \cite{Bollobas}.
We can thus define the Tutte function on the group $T$ by setting
$$
	T_g(x,y) := T_{\Gamma(T_1,T_2,c_-)}(x,y) (x+y)^{-n+1} , \quad g \in T
$$
where $T_1, T_2$ have $n$ leaves and $g=g(T_1,T_2,c_-)\in T$, cf. \cite{AiCo1}.

\begin{theorem}\label{pos-def-tutte}
The Tutte function $T_g(x,y)$ is of positive type on $T$ for $y=e^K$, $x=\frac{y+Q-1}{y-1}$ and $K\neq 0$.
\end{theorem}
\begin{proof}
Let $g_1, \ldots, g_r \in T$.
Adding pairs of opposing carets whenever necessary, we can always suppose that  $g_i=g(T_+^i,T_-, c^i_-)$  where $T_+^i$, $i=1,\ldots,r$ and $T_-$ have $n$ leaves. Hence, $g_ig_j^{-1}=g(T^i_+,T^j_+,c^i_- - c^j_- +n+1)$ and we need to show that the matrix 
$$\Big(T_{\Gamma(T^i_+,T^j_+,c^i_- - c^j_- +n+1)}(x,y)/(x+y)^{n-1}\Big)_{i,j=1}^r$$
is positive semi-definite. 
As in \cite{AiCo1}, it is enough to show that the matrix
$$(Z(\Gamma(T^i_+,T^j_+,c^i_- - c^j_- +n+1);Q,K))_{i,j=1}^r$$ is positive semi-definite, 
where $Z(G;Q,K)$ denotes the partition function of the Potts model on the graph $G$ (see e.g. \cite{Welsh}). 

The leaves of the trees $T^i_+$ come with a natural cyclic labelling by elements of $\{1,\ldots,n\}$ induced by numbering the leaves of $T_-$
in increasing order (recall that we join the first leaf of $T^i_+$ with the $c^i_-$ -th leaf of $T_-$).
The vertices of the associated graphs $\Gamma^i:=\Gamma_+(T^i_+,c^i_-)$ come with a natural labelling by elements of $\{1,\ldots,n\}$ in increasing order.  When computing the product $g_ig_j^{-1}$, the associated $\Gamma$-graph  $\Gamma(T^i_+,T^j_+,c^i_- - c^j_- +n+1)$  is obtained, up to a graph isomorphism induced by a cyclic permutation of the vertices, by glueing  $\Gamma^i$ and $\Gamma^j$ 
 (this is a consequence of the rules for $c_-$).

Since the Tutte polynomial/Potts partition function is invariant under graph isomorphisms, it remains to show that the matrix
$$(Z((\Gamma^i,\Gamma^j);Q,K))_{i,j=1}^r $$
is positive semi-definite.
But 
$$Z((\Gamma^i,\Gamma^j));Q,K)=\sum_\sigma e^{-K\sum_{\overline{ij}\in E(\Gamma^i)} (1-\delta(\sigma_i,\sigma_j))}e^{-K\sum_{\overline{ij}\in E(\Gamma^j)} (1-\delta(\sigma_i,\sigma_j))}.
$$
For any $\sigma=(\sigma_1,\cdots ,\sigma_n)$, the expression $e^{-K\sum_{\overline{ij}\in E(\Gamma^i)} (1-\delta(\sigma_i,\sigma_j))}$ defines the $\sigma$-th component of a vector $v_i$ in $\H=\mathbb{C}^{Q^n}$, $i=1,\ldots,r$ and
$$
	Z(\Gamma(T_i,T_j);Q,K)=\langle v_i ,v_j\rangle.
$$
It follows that the matrix $(T_{g_ig_j^{-1}}(x,y))_{i,j=1}^r$ is positive semi-definite for any $r$, i.e. the function $T_g(x,y)$ is of positive type on $T$.
\end{proof}

As in \cite{AiCo1}, we define a Kauffman bracket function on the Thompson group $T$. We recall that the Kauffman bracket (\cite{Kauffman}) is defined by the following skein-relation
$$
		 \left\{\begin{array}{l}
 		\langle \slashoverback\rangle=A\langle \smoothing \rangle+A^{-1}\langle \hsmoothing \rangle\\
 		\langle O\rangle= 1.
 		\end{array}
 		\right.
$$
We note that this polynomial is invariant under regular isotopies. Since the addition of a pair of opposite carets yields a distant unknot, the Kauffman bracket function is defined as
$$
	\langle g\rangle (A) := (-A^2-A^{-2})^{-n}\langle \L(T_1,T_2,c_-)\rangle (A)
$$
where $T_1, T_2$ have $n$ leaves and $g=g(T_+,T_-,c_-)\in T$.

We recall some results proved in \cite{DLH-J}. Given a signed graph $G$. Denote by $G^+$ and $G^-$ the subgraphs whose edges are the positive and the negative edges, respectively. For any $i, j\in V(G)$, define the function
\begin{eqnarray*}
	w(\sigma_i,\sigma_j)&=& \left\{\begin{array}{cc}
 		-A^3 & \textrm{ if $\sigma_i=\sigma_j$} \\ 
 		A^{-1} & \textrm{ if $\sigma_i\neq \sigma_j$} \\ 
 		\end{array}
 		\right.
\end{eqnarray*}
where $\sigma_i$ is the spin at site $i$. Set $w_+(\sigma_i,\sigma_j)=w(\sigma_i,\sigma_j)$ and $w_-(\sigma_i,\sigma_j)=w(\sigma_i,\sigma_j)^{-1}$. 

Consider the partition function defined by 
$$
	Z_{G}(A)=\left(\frac{1}{\sqrt{Q}}\right)^{|V(G)|+1}\sum_\sigma \prod_{\overline{ij}\in E(G^+)} w_+(\sigma_i,\sigma_j)\prod_{\overline{ij}\in E(G^-)}w_-(\sigma_i,\sigma_j) \ ,
$$
where the sum over $\sigma$ runs over all the spin configurations $\{1,\cdots,Q\}^{|V(G)|}$. 

Given a link $L$ with link diagram $D$, we denote by $F(D)$ its face graph (for a definition see \cite{Godsil}, p. 379).
Then the following identity holds
\begin{equation}\label{kaufpartfun}
	\langle L(D)\rangle=Z_{F(D)}.
\end{equation}
We also notice that, for any element in $T$, the corresponding face graph and $\Gamma-$graph coincide, namely
$$
	F(\L(T_+,T_-,c_-))=\Gamma(T_+,T_-,c_-) \ .
$$

\begin{theorem} \label{pos-def-kauff}
The function $\langle g\rangle(A)$, where $A$ is any solution of $A^2 + A^{-2} + \sqrt Q = 0$ for $Q=2, 3, 4$, is of positive type on $T$.
\end{theorem}
\begin{proof}
As in the proof of Theorem \ref{pos-def-tutte}, 
without loss of generality we can suppose that $g_i=g(T_+^i,T_-,c_-^i)$  where $T_+^i$, $i=1, \cdots, r$ and $T_-$ have $n$ leaves  and $g_ig_j^{-1}=g(T^i_+,T^j_+,c_-^i-c_-^j+n+1)$. Therefore, it is enough to consider  $\Big(\langle \L(T^i_+,T^j_+,c_-^i-c_-^j+n+1)\rangle /(-A^2-A^{-2})^{n}\Big)_{i,j=1}^r$ and, after neglecting a positive factor, prove that $(\langle \L(T^i_+,T^j_+,c_-^i-c_-^j+n+1)\rangle)_{i,j=1}^r$ is positive semi-definite. We can use the partition function defined above in order to prove our claim, because of the identity (\ref{kaufpartfun}) . 
Again, it is possible to consider a family of graphs $\Gamma^i$ such that $\Gamma(T^i_+,T^j,_+,c_-^i-c_-^j+n+1)$ is (graph) isomorphic to  $(\Gamma^i,\Gamma^j)$. Since 
the partition function depends only on the isomorphism class of the graph, we have that 
$$
	Z(\Gamma(T_+^i,T_+^j,c_-^i-c_-^j+n+1),x,y)=Z((\Gamma^i,\Gamma^j))=\langle v_{T_+^i},v_{T_+^j}\rangle
$$
where, for each $i \in \{1,\ldots,r\}$ and any $\sigma=(\sigma_1,\cdots ,\sigma_n)$, the expression $\prod_{\overline{ij}\in E(\Gamma^i)} w_+(\sigma_i,\sigma_j)$ defines the $\sigma$-th component of a vector $v_{T_+^i}$ in $\H=\mathbb{C}^{Q^n}$, i.e. the component corresponding to $e_{\sigma_1}\otimes \cdots\otimes e_{\sigma_n}$. Therefore, 
the matrix $(\langle \L(T^i_+,T^j_+,c_-^i-c_-^j+n+1)\rangle(x,y))_{i,j=1}^r$ is positive definite for any $r$, i.e. the function $\langle g\rangle$ is of positive type.
\end{proof}

\section{On the Jones polynomial as a function on  $\vec{T}$}
Jones introduced the oriented version of the group $T$, namely
$$
\vec{T}:=\{g \in T \ | {\rm Chr}_{\Gamma(g)}(2) = 2 \}.
$$
It is natural to ask whether some of the results proved for $\vec{F}$ also hold for $\vec{T}$. Clearly, the writhe function ${\rm wr}$ extends to $\vec{T}$ in the obvious way. First of all, we give a result similar to Proposition \ref{wr-hom-F}.
\begin{proposition}
For any $g, h \in\vec{T}$, we have that ${\rm wr}(gh)={\rm wr}(g)+{\rm wr}(h)$.
\end{proposition}
\begin{proof}
The difference with the procedure used for $\vec{F}$ is that in this case one has to consider the graphs $\Gamma_+$ and $\Gamma_-$, then rearrange the graph $\Gamma_+$ and eventually draw a link in the usual manner. We want to prove that for any rearrangement of the graph $\Gamma_+$, the writhe of the semi-link is always the same. This follows from the following observations. First of all we recall that each crossing corresponds to an edge of the $\Gamma_+$ graph. The type of the (oriented) crossing depends on the colours chosen for the two vertices connected by the edge. Consider an edge with left vertex coloured with $+$ and right vertex $-$. Then the corresponding crossing in the $\Gamma_+$ graph is $\overcrossing$. If in the new rearrangement the vertices have the same colours, then clearly the crossing is still positive. Otherwise we obtain the crossing $\vin$ and thus we still have a positive crossing. Now we have proved that the rearrangement does not effect the writhe and the main statement follows from an argument similar to the one used in the proof of Proposition \ref{wr-hom-F}.
\end{proof}
\begin{corollary}
The writhe function is identically zero on $\vec{T}$.
\end{corollary}
\begin{proof}
In the proof of the above Proposition it was shown that in the upper (lower) semi-link there are only crossings of positive (negative) type. This implies that the writhe of the link is zero.
\end{proof}
The proof of the following result is omitted since the argument is the same used in Proposition \ref{lemma-comp-link-F}.
\begin{proposition}
We have that $(-1)^{c(g)}\equiv 1$ on $\vec{T}$.
\end{proposition}

Now,  basically by the same proofs as for $\vec{F}$,  we obtain the following result.
\begin{theorem} \label{teo-fin-Jones-T}
The evaluations of the function $V_g(t)$ at $t = 1, i, e^{\pm i \pi/3}$ are of positive type on $\vec{T}$.
\end{theorem}

As in \cite{AiCo1}, it is possible to define the colouring function ${\rm Col}_Q(g)$ on $\vec{T}$. The above Theorem and the fact that ${\rm Col}_3(L)=3|V_{\vec{L}}(e^{i\pi/3})|^2$ then show that ${\rm Col}_3(g)$ is a function of positive type on $\vec{T}$.

\appendix
\section{Appendix}
In this appendix we compute all the possible oriented links coming from the graphs $\Gamma$ with $5$ vertices. 
The following graphs are all the possible $\Gamma_\pm-$graphs
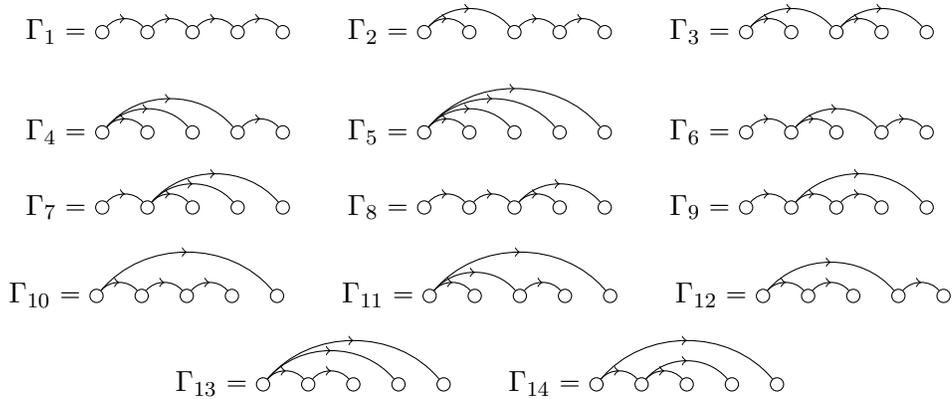
\[\Gamma_1=\begin{tikzpicture}[x=0.6cm, y=0.6cm,
    every edge/.style={
        draw,
        postaction={decorate,
                    decoration={markings,mark=at position 0.5 with {\arrow{>}}}
                   }
        }
]
	\vertex (a) at (0,0) {};
	\vertex (b) at (1,0) {};
	\vertex (c) at (2,0) {};
	\vertex (d) at (3,0) {};
	\vertex (e) at (4,0) {};
	\path
		(a) edge[bend left=50] (b)
		(b) edge[bend left=50] (c)
		(c) edge[bend left=50] (d)
		(d) edge[bend left=50] (e)
;
\end{tikzpicture}
\qquad
\Gamma_2=\begin{tikzpicture}[x=0.6cm, y=0.6cm,
    every edge/.style={
        draw,
        postaction={decorate,
                    decoration={markings,mark=at position 0.5 with {\arrow{>}}}
                   }
        }
]
	\vertex (a) at (0,0) {};
	\vertex (b) at (1,0) {};
	\vertex (c) at (2,0) {};
	\vertex (d) at (3,0) {};
	\vertex (e) at (4,0) {};
	\path
		(a) edge[bend left=50] (b)
		(a) edge[bend left=50] (c)
		(c) edge[bend left=50] (d)
		(d) edge[bend left=50] (e)
;
\end{tikzpicture}
\qquad
\Gamma_3=\begin{tikzpicture}[x=0.6cm, y=0.6cm,
    every edge/.style={
        draw,
        postaction={decorate,
                    decoration={markings,mark=at position 0.5 with {\arrow{>}}}
                   }
        }
]
	\vertex (a) at (0,0) {};
	\vertex (b) at (1,0) {};
	\vertex (c) at (2,0) {};
	\vertex (d) at (3,0) {};
	\vertex (e) at (4,0) {};
	\path
		(a) edge[bend left=50] (b)
		(a) edge[bend left=50] (c)
		(c) edge[bend left=50] (d)
		(c) edge[bend left=50] (e)
;
\end{tikzpicture}\]
\[\Gamma_4=\begin{tikzpicture}[x=0.6cm, y=0.6cm,
    every edge/.style={
        draw,
        postaction={decorate,
                    decoration={markings,mark=at position 0.5 with {\arrow{>}}}
                   }
        }
]
	\vertex (a) at (0,0) {};
	\vertex (b) at (1,0) {};
	\vertex (c) at (2,0) {};
	\vertex (d) at (3,0) {};
	\vertex (e) at (4,0) {};
	\path
		(a) edge[bend left=50] (b)
		(a) edge[bend left=50] (c)
		(a) edge[bend left=50] (d)
		(d) edge[bend left=50] (e)
;
\end{tikzpicture}
\qquad
\Gamma_5=\begin{tikzpicture}[x=0.6cm, y=0.6cm,
    every edge/.style={
        draw,
        postaction={decorate,
                    decoration={markings,mark=at position 0.5 with {\arrow{>}}}
                   }
        }
]
	\vertex (a) at (0,0) {};
	\vertex (b) at (1,0) {};
	\vertex (c) at (2,0) {};
	\vertex (d) at (3,0) {};
	\vertex (e) at (4,0) {};
	\path
		(a) edge[bend left=50] (b)
		(a) edge[bend left=50] (c)
		(a) edge[bend left=50] (d)
		(a) edge[bend left=50] (e)
;
\end{tikzpicture}
\qquad
\Gamma_6=\begin{tikzpicture}[x=0.6cm, y=0.6cm,
    every edge/.style={
        draw,
        postaction={decorate,
                    decoration={markings,mark=at position 0.5 with {\arrow{>}}}
                   }
        }
]
	\vertex (a) at (0,0) {};
	\vertex (b) at (1,0) {};
	\vertex (c) at (2,0) {};
	\vertex (d) at (3,0) {};
	\vertex (e) at (4,0) {};
	\path
		(a) edge[bend left=50] (b)
		(b) edge[bend left=50] (c)
		(b) edge[bend left=50] (d)
		(d) edge[bend left=50] (e)
;
\end{tikzpicture}\]
\[\Gamma_7=\begin{tikzpicture}[x=0.6cm, y=0.6cm,
    every edge/.style={
        draw,
        postaction={decorate,
                    decoration={markings,mark=at position 0.5 with {\arrow{>}}}
                   }
        }
]
	\vertex (a) at (0,0) {};
	\vertex (b) at (1,0) {};
	\vertex (c) at (2,0) {};
	\vertex (d) at (3,0) {};
	\vertex (e) at (4,0) {};
	\path
		(a) edge[bend left=50] (b)
		(b) edge[bend left=50] (c)
		(b) edge[bend left=50] (d)
		(b) edge[bend left=50] (e)
;
\end{tikzpicture}
\qquad
\Gamma_8=\begin{tikzpicture}[x=0.6cm, y=0.6cm,
    every edge/.style={
        draw,
        postaction={decorate,
                    decoration={markings,mark=at position 0.5 with {\arrow{>}}}
                   }
        }
]
	\vertex (a) at (0,0) {};
	\vertex (b) at (1,0) {};
	\vertex (c) at (2,0) {};
	\vertex (d) at (3,0) {};
	\vertex (e) at (4,0) {};
	\path
		(a) edge[bend left=50] (b)
		(b) edge[bend left=50] (c)
		(c) edge[bend left=50] (d)
		(c) edge[bend left=50] (e)
;
\end{tikzpicture}
\qquad 
\Gamma_9=\begin{tikzpicture}[x=0.6cm, y=0.6cm,
    every edge/.style={
        draw,
        postaction={decorate,
                    decoration={markings,mark=at position 0.5 with {\arrow{>}}}
                   }
        }
]
	\vertex (a) at (0,0) {};
	\vertex (b) at (1,0) {};
	\vertex (c) at (2,0) {};
	\vertex (d) at (3,0) {};
	\vertex (e) at (4,0) {};
	\path
		(a) edge[bend left=50] (b)
		(b) edge[bend left=50] (c)
		(c) edge[bend left=50] (d)
		(b) edge[bend left=50] (e)
;
\end{tikzpicture}\]
\[
\Gamma_{10}=\begin{tikzpicture}[x=0.6cm, y=0.6cm,
    every edge/.style={
        draw,
        postaction={decorate,
                    decoration={markings,mark=at position 0.5 with {\arrow{>}}}
                   }
        }
]
	\vertex (a) at (0,0) {};
	\vertex (b) at (1,0) {};
	\vertex (c) at (2,0) {};
	\vertex (d) at (3,0) {};
	\vertex (e) at (4,0) {};
	\path
		(a) edge[bend left=50] (b)
		(b) edge[bend left=50] (c)
		(c) edge[bend left=50] (d)
		(a) edge[bend left=50] (e)
;
\end{tikzpicture}
\qquad 
\Gamma_{11}=\begin{tikzpicture}[x=0.6cm, y=0.6cm,
    every edge/.style={
        draw,
        postaction={decorate,
                    decoration={markings,mark=at position 0.5 with {\arrow{>}}}
                   }
        }
]
	\vertex (a) at (0,0) {};
	\vertex (b) at (1,0) {};
	\vertex (c) at (2,0) {};
	\vertex (d) at (3,0) {};
	\vertex (e) at (4,0) {};
	\path
		(a) edge[bend left=50] (b)
		(a) edge[bend left=50] (c)
		(c) edge[bend left=50] (d)
		(a) edge[bend left=50] (e)
;
\end{tikzpicture}
\qquad
\Gamma_{12}=\begin{tikzpicture}[x=0.6cm, y=0.6cm,
    every edge/.style={
        draw,
        postaction={decorate,
                    decoration={markings,mark=at position 0.5 with {\arrow{>}}}
                   }
        }
]
	\vertex (a) at (0,0) {};
	\vertex (b) at (1,0) {};
	\vertex (c) at (2,0) {};
	\vertex (d) at (3,0) {};
	\vertex (e) at (4,0) {};
	\path
		(a) edge[bend left=50] (b)
		(b) edge[bend left=50] (c)
		(a) edge[bend left=50] (d)
		(d) edge[bend left=50] (e)
;
\end{tikzpicture}\]
\[
\Gamma_{13}=\begin{tikzpicture}[x=0.6cm, y=0.6cm,
    every edge/.style={
        draw,
        postaction={decorate,
                    decoration={markings,mark=at position 0.5 with {\arrow{>}}}
                   }
        }
]
	\vertex (a) at (0,0) {};
	\vertex (b) at (1,0) {};
	\vertex (c) at (2,0) {};
	\vertex (d) at (3,0) {};
	\vertex (e) at (4,0) {};
	\path
		(a) edge[bend left=50] (b)
		(b) edge[bend left=50] (c)
		(a) edge[bend left=50] (d)
		(a) edge[bend left=50] (e)
;
\end{tikzpicture}
\qquad
\Gamma_{14}=\begin{tikzpicture}[x=0.6cm, y=0.6cm,
    every edge/.style={
        draw,
        postaction={decorate,
                    decoration={markings,mark=at position 0.5 with {\arrow{>}}}
                   }
        }
]
	\vertex (a) at (0,0) {};
	\vertex (b) at (1,0) {};
	\vertex (c) at (2,0) {};
	\vertex (d) at (3,0) {};
	\vertex (e) at (4,0) {};
	\path
		(a) edge[bend left=50] (b)
		(b) edge[bend left=50] (c)
		(b) edge[bend left=50] (d)
		(a) edge[bend left=50] (e)
;
\end{tikzpicture}
\]

If the pair $(\Gamma_i,\Gamma_j)$ is bipartite, we denote by $g_{ij}$ the corresponding element in $\vec{F}$. These are the links associated to the elements $g_{ij}\in \vec{F}$ 

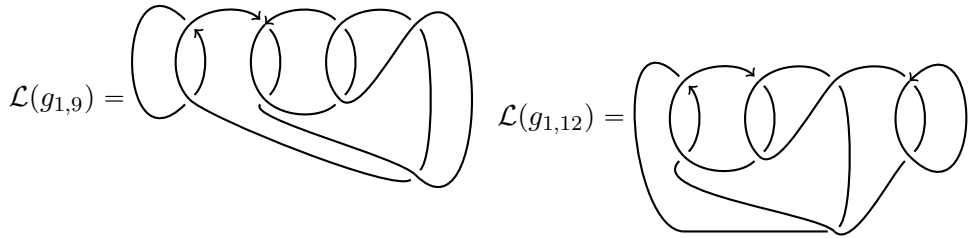
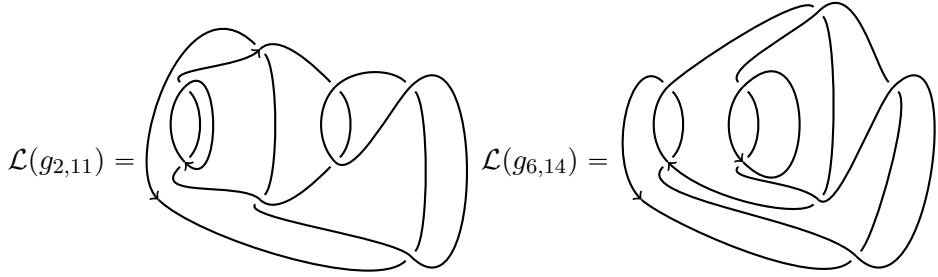
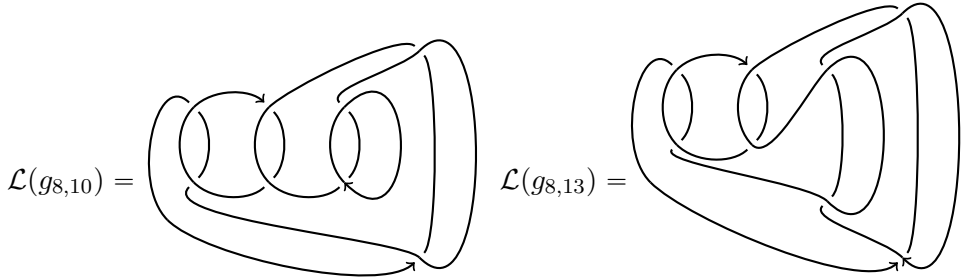
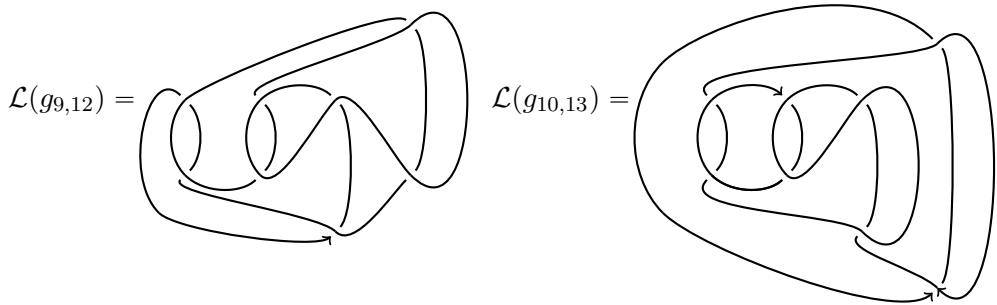
\[\begin{tikzpicture}[every path/.style={thick}, every node/.style={transform shape, knot crossing, inner sep=1.5pt}]
	\node (a) at (0,0) {};
	\node (b) at (1,0) {};
	\node (c) at (2,0) {};
	\node (d) at (3,0) {};
	\node (e) at (0,-1) {};
	\node (f) at (1,-1) {};
	\node (g) at (2,-1) {};
	\node (h) at (3,-2) {};
	\node (x) at (-1.5,-1) {$\L(g_{1,9})= \  \ $};

\draw (a) .. controls (a.16 north west) and (e.16 south west) ..  (e);
\draw (a.center) .. controls (a.4 south west) and (e.4 north west) ..  (e.center);
\draw[<-] (a) .. controls (a.4 south east) and (e.4 north east) ..  (e);

\draw (b.center) .. controls (b.4 south west) and (f.4 north west) ..  (f.center);
\draw (b) .. controls (b.4 south east) and (f.4 north east) ..  (f);

\draw (c.center) .. controls (c.4 south west) and (g.4 north west) ..  (g.center);
\draw (c) .. controls (c.4 south east) and (g.4 north east) ..  (g);

\draw (d.center) .. controls (d.4 south west) and (g.4 south east) ..  (g.center);
\draw (d) .. controls (d.4 south east) and (h.4 north east) ..  (h);

\draw[->] (a.center) .. controls (a.4 north east) and (b.4 north west) ..  (b);
\draw[<-]  (b.center) .. controls (b.4 north east) and (c.4 north west) ..  (c);
\draw (c.center) .. controls (c.4 north east) and (d.4 north west) ..  (d);

\draw (d.center) .. controls (d.16 north east) and (h.16 south east) ..  (h.center);

\draw (e.center) .. controls (e.4 south east) and (h.4 south west) ..  (h);

\draw (h.center) .. controls (h.4 north west) and (f.4 south west) ..  (f);

\draw (f.center) .. controls (f.4 south east) and (g.4 south west) ..  (g);
\end{tikzpicture}
\qquad \qquad
\begin{tikzpicture}[every path/.style={thick}, every node/.style={transform shape, knot crossing, inner sep=1.5pt}]
	\node (a) at (0,0) {};
	\node (b) at (1,0) {};
	\node (c) at (2,0) {};
	\node (d) at (3,0) {};
	\node (e) at (0,-1) {};
	\node (f) at (1,-1) {};
	\node (g) at (2,-2) {};
	\node (h) at (3,-1) {};
	\node (x) at (-1.5,-0.5) {$\L(g_{1,12})= \ \ $};

	\node (ap) at (0,-2) {};

\draw (a) .. controls (a.16 north west) and (ap.8 west) .. (ap.center);
\draw (ap.center) .. controls (ap.8  east) and (g.8  west) ..  (g);

\draw (a.center) .. controls (a.4 south west) and (e.4 north west) ..  (e.center);
\draw[<-] (a) .. controls (a.4 south east) and (e.4 north east) ..  (e);

\draw (b.center) .. controls (b.4 south west) and (f.4 north west) ..  (f.center);
\draw (b) .. controls (b.4 south east) and (f.4 north east) ..  (f);


\draw (c) .. controls (c.4 south east) and (g.4 north east) ..  (g);
\draw (d) .. controls (d.4 south east) and (h.4 north east) ..  (h);

\draw[->]  (a.center) .. controls (a.4 north east) and (b.4 north west) ..  (b);
\draw (b.center) .. controls (b.4 north east) and (c.4 north west) ..  (c);
\draw (c.center) .. controls (c.4 north east) and (d.4 north west) ..  (d);

\draw[<-] (d.center) .. controls (d.16 north east) and (h.16 south east) ..  (h.center);

\draw (e) .. controls (e.8 south west) and (g.4 north west) ..  (g.center);

\draw (e.center) .. controls (e.4 south east) and (f.4 south west) ..  (f);

\draw (g.center) .. controls (g.4 south east) and (h.4 south west) ..  (h);

\draw (h.center) .. controls (h.4 north west) and (d.4 south west) ..  (d.center);

\draw (f.center) .. controls (f.4 south east) and (c.4 south west) ..  (c.center);
\end{tikzpicture}\] 
\[
\begin{tikzpicture}[every path/.style={thick}, every node/.style={transform shape, knot crossing, inner sep=1.5pt}]
	\node (a) at (0,0) {};
	\node (b) at (1,.5) {};
	\node (c) at (2,0) {};
	\node (d) at (3,0) {};
	\node (e) at (0,-1) {};
	\node (f) at (1,-1.5) {};
	\node (g) at (2,-1) {};
	\node (h) at (3,-2.25) {};
	\node (x) at (-1.5,-1) {$\L(g_{2,11})=$};

	\node (ap) at (-0.35,-1.5) {};

\draw[->] (a) .. controls (a.4 north west) and (b.4 south west) .. (b.center);
\draw (a.center) .. controls (a.4 south west) and (e.4 north west) ..  (e.center);
\draw[->] (a.center) .. controls (a.8 north east) and (e.8 south east) ..  (e.center);

\draw (a) .. controls (a.4 south east) and (e.4 north east) ..  (e);

\draw (b) .. controls (b.4 south east) and (f.4 north east) ..  (f);
\draw (b.center) .. controls (b.4 north east) and (c.4 north west) ..  (c);

\draw (c.center) .. controls (c.4 south west) and (g.4 north west) ..  (g.center);
\draw (c) .. controls (c.4 south east) and (g.4 north east) ..  (g);

\draw (c.center) .. controls (c.4 north east) and (d.4 north west) ..  (d);
\draw (d.center) .. controls (d.4 south west) and (g.4 south east) ..  (g.center);

\draw (d) .. controls (d.4 south east) and (h.4 north east) ..  (h);
\draw (d.center) .. controls (d.16 north east) and (h.16 south east) ..  (h.center);
\draw (h.center) .. controls (h.4 north west) and (f.4 south west) ..  (f);

\draw (f.center) .. controls (f.4 south east) and (g.4 south west) ..  (g);

\draw (e) .. controls (e.8 south west) and (f.4 north west) ..  (f.center);

\draw[->] (b) .. controls (b.16 north west) and (ap.8 north west) ..  (ap.center);
\draw (ap.center) .. controls (ap.8 south east) and (h.8 south west) ..  (h);

\end{tikzpicture}
\qquad \qquad
\begin{tikzpicture}[every path/.style={thick}, every node/.style={transform shape, knot crossing, inner sep=1.5pt}]
	\node (a) at (0,0) {};
	\node (b) at (1,0) {};
	\node (c) at (2,1) {};
	\node (d) at (3,0) {};
	\node (e) at (0,-1) {};
	\node (f) at (1,-1) {};
	\node (g) at (2,-1.5) {};
	\node (h) at (2.5,-2.25) {};
	\node (x) at (-1.5,-1) {$\L(g_{6,14})= \ \ $};

	\node (ap) at (-0.35,-1.5) {};

\draw (a.center) .. controls (a.4 north east) and (c.4 north west) .. (c);
\draw (a.center) .. controls (a.4 south west) and (e.4 north west) ..  (e.center);
\draw (a) .. controls (a.4 south east) and (e.4 north east) ..  (e);

\draw (b) .. controls (b.4 north west) and (c.4 south west) ..  (c.center);

\draw (b) .. controls (b.4 south east) and (f.4 north east) ..  (f);
\draw[->] (b.center) .. controls (b.4 south west) and (f.4 north west) ..  (f.center);
\draw (b.center) .. controls (b.16 north east) and (f.16 south east) ..  (f.center);

\draw (c.center) .. controls (c.8 north east) and (d.4 north west) ..  (d);
\draw (c) .. controls (c.4 south east) and (g.4 north east) ..  (g);

\draw (d.center) .. controls (d.4 south west) and (g.4 south east) ..  (g.center);

\draw (d) .. controls (d.4 south east) and (h.4 north east) ..  (h);
\draw (d.center) .. controls (d.16 north east) and (h.16 south east) ..  (h.center);
\draw (h.center) .. controls (h.8 north west) and (e.8 south west) ..  (e);

\draw (f) .. controls (f.4 south west) and (g.4 north west) ..  (g.center);

\draw[<-] (e.center) .. controls (e.8 south east) and (g.4 south west) ..  (g);
\draw[->] (a) .. controls (a.8 north west) and (ap.8 north west) ..  (ap.center);
\draw (ap.center) .. controls (ap.4 south east) and (h.8 south west) ..  (h);

\end{tikzpicture}\] 
\[
\begin{tikzpicture}[every path/.style={thick}, every node/.style={transform shape, knot crossing, inner sep=1.5pt}]
	\node (a) at (0,0) {};
	\node (b) at (1,0) {};
	\node (c) at (2,0) {};
	\node (d) at (3,.75) {};
	\node (e) at (0,-1) {};
	\node (f) at (1,-1) {};
	\node (g) at (2,-1) {};
	\node (h) at (3,-2) {};
	\node (x) at (-1.5,-1) {$\L(g_{8,10})= \ \ $};

	\node (ap) at (-0.35,-1.5) {};

\draw[->] (a.center) .. controls (a.4 north east) and (b.4 north west) .. (b);
\draw (a.center) .. controls (a.4 south west) and (e.4 north west) ..  (e.center);
\draw (a) .. controls (a.4 south east) and (e.4 north east) ..  (e);

\draw (b.center) .. controls (b.4 north east) and (d.4 north west) ..  (d);

\draw (b) .. controls (b.4 south east) and (f.4 north east) ..  (f);
\draw (b.center) .. controls (b.4 south west) and (f.4 north west) ..  (f.center);

\draw[->] (c.center) .. controls (c.16 north east) and (g.16 south east) ..  (g.center);
\draw (c) .. controls (c.4 south east) and (g.4 north east) ..  (g);
\draw (c.center) .. controls (c.4 south west) and (g.4 north west) ..  (g.center);
\draw (c) .. controls (c.4 north west) and (d.4 south west) ..  (d.center);

\draw (d.center) .. controls (d.16 north east) and (h.16 south east) ..  (h.center);
\draw (d) .. controls (d.4 south east) and (h.4 north east) ..  (h);

\draw (e.center) .. controls (e.4 south east) and (f.4 south west) ..  (f);
\draw (e) .. controls (e.8 south west) and (h.4 north west) ..  (h.center);
\draw[<-] (h) .. controls (h.8 south west) and (ap.8 south east) ..  (ap.center);
\draw (ap.center) .. controls (ap.8 north west) and (a.8 north west) ..  (a);

\draw (f.center) .. controls (f.4 south east) and (g.4 south west) ..  (g);

\end{tikzpicture}
\qquad \qquad
\begin{tikzpicture}[every path/.style={thick}, every node/.style={transform shape, knot crossing, inner sep=1.5pt}]
	\node (a) at (0,0) {};
	\node (b) at (1,0) {};
	\node (c) at (2,0) {};
	\node (d) at (3,.75) {};
	\node (e) at (0,-1) {};
	\node (f) at (1,-1) {};
	\node (g) at (2,-1.75) {};
	\node (h) at (3,-2.5) {};
	\node (x) at (-1.5,-1.5) {$\L(g_{8,13})=$};

	\node (ap) at (-0.35,-1.5) {};

\draw[->] (a.center) .. controls (a.4 north east) and (b.4 north west) .. (b);
\draw (a.center) .. controls (a.4 south west) and (e.4 north west) ..  (e.center);
\draw (a) .. controls (a.4 south east) and (e.4 north east) ..  (e);

\draw (b.center) .. controls (b.4 north east) and (d.4 north west) ..  (d);

\draw (b) .. controls (b.4 south east) and (f.4 north east) ..  (f);
\draw (b.center) .. controls (b.4 south west) and (f.4 north west) ..  (f.center);

\draw (c.center) .. controls (c.16 north east) and (g.16 south east) ..  (g.center);
\draw (c) .. controls (c.4 south east) and (g.4 north east) ..  (g);
\draw (c.center) .. controls (c.4 south west) and (f.4 south east) ..  (f.center);
\draw (c) .. controls (c.4 north west) and (d.4 south west) ..  (d.center);

\draw[->] (d.center) .. controls (d.16 north east) and (h.16 south east) ..  (h.center);
\draw (d) .. controls (d.4 south east) and (h.4 north east) ..  (h);

\draw (e.center) .. controls (e.4 south east) and (f.4 south west) ..  (f);
\draw (e) .. controls (e.4 south west) and (g.4 north west) ..  (g.center);
\draw[<-] (h) .. controls (h.8 south west) and (ap.8 south east) ..  (ap.center);
\draw (ap.center) .. controls (ap.8 north west) and (a.8 north west) ..  (a);

\draw (g) .. controls (g.4 south west) and (h.4 north west) ..  (h.center);

\end{tikzpicture}
\] 
\[
\begin{tikzpicture}[every path/.style={thick}, every node/.style={transform shape, knot crossing, inner sep=1.5pt}]
	\node (a) at (0,0) {};
	\node (b) at (1,0) {};
	\node (c) at (2,0) {};
	\node (d) at (3,1) {};
	\node (e) at (0,-1) {};
	\node (f) at (1,-1) {};
	\node (g) at (2,-1.75) {};
	\node (h) at (3,-1) {};
	\node (x) at (-1.5,0) {$\L(g_{9,12})=$};

	\node (ap) at (-0.35,-1.5) {};

\draw (a.center) .. controls (a.4 north east) and (d.4 north west) .. (d);
\draw (a.center) .. controls (a.4 south west) and (e.4 north west) ..  (e.center);
\draw (a) .. controls (a.4 south east) and (e.4 north east) ..  (e);

\draw (b.center) .. controls (b.4 north east) and (c.4 north west) ..  (c);
\draw (b) .. controls (b.4 south east) and (f.4 north east) ..  (f);
\draw (b.center) .. controls (b.4 south west) and (f.4 north west) ..  (f.center);
\draw (b) .. controls (b.4 north west) and (d.4 south west) ..  (d.center);

\draw (c.center) .. controls (c.4 north east) and (h.4 north west) ..  (h.center);
\draw (c) .. controls (c.4 south east) and (g.4 north east) ..  (g);
\draw (c.center) .. controls (c.4 south west) and (f.4 south east) ..  (f.center);

\draw (d.center) .. controls (d.16 north east) and (h.16 south east) ..  (h.center);
\draw (d) .. controls (d.4 south east) and (h.4 north east) ..  (h);

\draw (e.center) .. controls (e.4 south east) and (f.4 south west) ..  (f);

\draw (g.center) .. controls (g.4 south east) and (h.4 south west) ..  (h);
\draw (g.center) .. controls (g.4 north west) and (e.4 south west) ..  (e);
\draw[<-] (g) .. controls (g.4 south west) and (ap.4 south east) ..  (ap.center);
\draw (ap.center) .. controls (ap.8 north west) and (a.8 north west) ..  (a);


%
	\node (a) at (7,0) {};
	\node (b) at (8,0) {};
	\node (c) at (9,0) {};
	\node (d) at (10,.75) {};
	\node (e) at (7,-1) {};
	\node (f) at (8,-1) {};
	\node (g) at (9,-1.75) {};
	\node (h) at (10,-2.5) {};
	\node (x) at (5,0) {$\L(g_{10,13})=$};

	\node (ap) at (6.35,.5) {};
	\node (ap2) at (6.35,-1.5) {};

%
%

\draw[->] (a.center) .. controls (a.4 north east) and (b.4 north west) .. (b);
\draw (a.center) .. controls (a.4 south west) and (e.4 north west) ..  (e.center);
\draw (a) .. controls (a.4 south east) and (e.4 north east) ..  (e);
\draw (a) .. controls (a.8 north west) and (d.4 south west) ..  (d.center);

\draw (b.center) .. controls (b.4 north east) and (c.4 north west) ..  (c);
\draw (b) .. controls (b.4 south east) and (f.4 north east) ..  (f);
\draw (b.center) .. controls (b.4 south west) and (f.4 north west) ..  (f.center);

\draw (c.center) .. controls (c.16 north east) and (g.16 south east) ..  (g.center);
\draw (c) .. controls (c.4 south east) and (g.4 north east) ..  (g);
\draw (c.center) .. controls (c.4 south west) and (f.4 south east) ..  (f.center);

\draw (e.center) .. controls (e.4 south east) and (f.4 south west) ..  (f);

\draw[->] (d.center) .. controls (d.16 north east) and (h.16 south east) ..  (h.center);
\draw (d) .. controls (d.4 south east) and (h.4 north east) ..  (h);

\draw (e.center) .. controls (e.4 south east) and (f.4 south west) ..  (f);

\draw (g.center) .. controls (g.4 north west) and (e.8 south west) ..  (e);
\draw (g) .. controls (g.4 south west) and (h.4 north west) ..  (h.center);

\draw[<-] (h) .. controls (h.8 south west) and (ap2.8 south east) ..  (ap2.center);
\draw (ap2.center) .. controls (ap2.8 north west) and (ap.8 south west) ..  (ap.center);
\draw (ap.center) .. controls (ap.8 north east) and (d.16 north west) ..  (d);


\end{tikzpicture}
\]

The matrix $(c(g_{ij}))_{i,j=1,\ldots, 14}$ is clearly symmetric (see Proposition 2.4, \cite{AiCo1}) and has the following entries
\begin{itemize}
\item $c(g_{1,9})=c(g_{1,12})=c(g_{2,11})=c(g_{6,14})=c(g_{8,10})=c(g_{8,13})=c(g_{13,10})=-2$;
\item $c(g_{9,12})=-4$.
\end{itemize}

\bigskip
{\parindent=0pt 

\smallskip Valeriano Aiello\\ 
Section de Math\'ematiques \\
Universit\'e de Gen\`eve \\
2-4 rue du Li\`evre, Case Postale 64, 1211 Gen\`eve 4, Switzerland\\
E-mail: valerianoaiello@gmail.com\\

\smallskip \noindent
Roberto Conti\\
Dipartimento di Scienze di Base e Applicate per l'Ingegneria \\
Sapienza Universit\`a di Roma \\
Via A. Scarpa 16, I-00161 Roma, Italy
\\ E-mail: roberto.conti@sbai.uniroma1.it
\par}

\end{document}